\documentclass[a4paper,10pt,draft]{amsart} %

\usepackage{amssymb, amsmath, amsthm, latexsym, verbatim, enumerate}

\DeclareMathOperator{\id}{id}
\DeclareMathOperator{\Ker}{Ker \,}
\DeclareMathOperator{\Span}{Span}
\DeclareMathOperator{\Tr}{Tr \,}
\DeclareMathOperator{\SO}{SO}
\DeclareMathOperator{\SL}{SL}
\DeclareMathOperator{\GL}{GL}
\DeclareMathOperator{\diag}{diag}
\DeclareMathOperator{\Conv}{Conv}
\DeclareMathOperator{\Aut}{Aut}
\DeclareMathOperator{\ad}{ad}
\DeclareMathOperator{\ric}{Ric}
\DeclareMathOperator{\Der}{Der}
\DeclareMathOperator{\rk}{rk}
\DeclareMathOperator{\End}{End}

\newcommand \<{\langle}
\newcommand \ra{\rangle}
\newcommand \ip{\<\cdot,\cdot\ra}

\newcommand \al{\alpha}
\newcommand \la{\lambda}
\newcommand \bc{\mathbb{C}}
\newcommand \br{\mathbb{R}}
\newcommand \Rn{\mathbb R^n}

\newcommand \g{\mathfrak{g}}
\newcommand \slg{\mathfrak{sl}}
\newcommand \h{\mathfrak{h}}
\newcommand \n{\mathfrak{n}}
\newcommand \m{\mathfrak{m}}
\newcommand \z{\mathfrak{z}}
\newcommand \ag{\mathfrak{a}}

\newcommand \p{\mathfrak{p}}
\newcommand \f{\mathfrak{f}}
\newcommand \s{\mathfrak{s}}
\newcommand \dg{\mathfrak{d}}
\newcommand \tg{\mathfrak{t}}

\newtheorem{theorem}{Theorem}
\newtheorem{lemma}{Lemma}
\newtheorem{proposition}{Proposition}

\theoremstyle{definition}
\newtheorem{definition}{Definition}
\newtheorem{example}{Example} 
\newtheorem*{example*}{Example}

\theoremstyle{remark}
\newtheorem{remark}{Remark}

\begin{document}

\title{Einstein solvmanifolds and the pre-Einstein derivation}

\author{Y.Nikolayevsky}
\address{Department of Mathematics, La Trobe University, Victoria, 3086, Australia}
\email{y.nikolayevsky@latrobe.edu.au}


\date{}

\begin{abstract}
An Einstein nilradical is a nilpotent Lie algebra, which can be
the nilradical of a metric Einstein solvable Lie algebra. The classification of
Riemannian Einstein solvmanifolds (possibly, of all noncompact homogeneous Einstein spaces) can
be reduced to determining, which nilpotent Lie algebras are Einstein nilradicals and to finding,
for every Einstein nilradical, its Einstein metric solvable extension.
For every nilpotent Lie algebra, we construct an (essentially unique) derivation,
the pre-Einstein derivation, the solvable extension by which may carry an Einstein inner product.
Using the pre-Einstein derivation, we then give a variational
characterization of Einstein nilradicals. As an application, we prove an easy-to-check convex
geometry condition for a nilpotent Lie algebra with a nice basis to be an Einstein nilradical
and also show that a typical two-step nilpotent Lie algebra is an Einstein nilradical.
\end{abstract}

\maketitle

\section{Introduction}
\label{s:intro}

The theory of Riemannian homogeneous spaces with an Einstein metric splits into three very different cases
depending on the sign of the Einstein constant, the scalar curvature. Among them, the picture is complete only in the
Ricci-flat case: by the result of \cite{AK}, every Ricci-flat homogeneous space is flat.

The major open conjecture in the case of negative scalar curvature is the \emph{Alekseevski Conjecture} \cite{Al1}
asserting that a noncompact Einstein homogeneous space admits a simply transitive solvable isometry group. This is
equivalent to saying that any such space is a \emph{solvmanifold}, a solvable Lie group with a left-invariant Riemannian
metric satisfying the Einstein condition.

By a deep result of J.Lauret \cite{La5}, any Einstein solvmanifold is \emph{standard}. This means that the metric solvable
Lie algebra $\s$ of such a solvmanifold has the following property: the orthogonal complement to the derived algebra
of $\s$ is abelian. The systematic study of standard Einstein solvmanifolds (and the term ``standard") originated from
the paper of J.Heber \cite{Heb}.

On the Lie algebra level, all the metric Einstein solvable Lie algebras can be obtained as the result of the following
construction \cite{Heb, La1, La5, LW}. One starts with the three pieces of data: a nilpotent Lie algebra $\n$, a
semisimple derivation $\Phi$ of $\n$, and an inner product $\ip_\n$ on $\n$, with respect to which $\Phi$ is
symmetric. An extension of $\n$ by $\Phi$ is a solvable Lie algebra $\s=\br H \oplus \n$ (as a linear space) with
$(\ad_H)_{|\n}:=\Phi$. The inner product on $\s$ is defined by $\<H, \n\ra = 0$, $\|H\|^2 = \Tr \Phi$ (and coincides with
the existing one on $\n$). The resulting metric solvable Lie algebra $(\s,\ip)$ is Einstein provided
$\n$ is ``nice" and the derivation $\Phi$ and the inner product $\ip_\n$ are
chosen ``in the correct way" (note, however, that these conditions are expressed by a system of algebraic equations,
which could hardly be analyzed directly, see Section~\ref{s:facts}). Metric Einstein solvable Lie algebras of higher rank
(with the codimension of the nilradical greater than one) having the same nilradical $\n$ can be obtained from $\s$ via
a known procedure, by further adjoining to $\n$ semisimple derivation commuting with $\Phi$.

It turns out that the structure of an Einstein metric solvable Lie algebra is completely encoded in its nilradical in
the following sense: given a nilpotent Lie algebra $\n$, there is no more than one (possibly none) choice of $\Phi$ and
of $\ip_\n$, up to conjugation by $\Aut(\n)$ and scaling, which may result in an Einstein metric solvable Lie
algebra $(\s,\ip)$.

\begin{definition} \label{d:en}
A nilpotent Lie algebra is called an \emph{Einstein nilradical}, if it is the nilradical of an Einstein metric solvable
Lie algebra. A derivation $\Phi$ of an Einstein nilradical $\n$ and an inner product $\ip_\n$, for which
the metric solvable Lie algebra $(\s,\ip)$ is Einstein are called an \emph{Einstein derivation} and a
\emph{nilsoliton} inner product respectively.
\end{definition}

In this paper, we address the following two questions:
\begin{enumerate} [(A)]
  \item \label{q:A}
  How to determine, whether a given nilpotent Lie algebra $\n$ is an Einstein nilradical?
  \item  \label{q:B}
  If $\n$ is an Einstein nilradical, how to construct an Einstein solvmanifold whose Lie algebra
  has $\n$ as its nilradical?
\end{enumerate}

To answer question~\eqref{q:B}, we have to produce an Einstein derivation and a nilsoliton inner product for~$\n$.
For an Einstein derivation, the answer is given by Theorem~\ref{t:preE} below: any Einstein derivation
is a positive multiple of a pre-Einstein derivation (which in practice can be found by solving a system of linear
equations). The nilsoliton inner product could rarely be found explicitly,
unless $\n$ has a very simple structure. Implicitly it is characterized by (iii) of Theorem~\ref{t:var} and
Remark~\ref{rem:critical} below (see also Theorem~\ref{t:varm} in Section~\ref{s:thpE}).

Question~\eqref{q:A} is much more delicate. A necessary condition for a nilpotent Lie algebra to be an Einstein
nilradical
is that it admits an $\mathbb{N}$-gradation (which is defined by the Einstein derivation \cite{Heb}). However,
not every  $\mathbb{N}$-graded nilpotent Lie algebra is an Einstein nilradical.

It is known, for instance, that the following nilpotent Lie algebras $\n$ are Einstein nilradicals:
$\n$ is abelian \cite{Al2}, $\n$ has a codimension one abelian ideal \cite{La2},
$\dim \n \le 6$ \cite{Wil, La2}, $\n$ is the algebra of strictly upper-triangular matrices \cite{Pay}.
Free Einstein nilradical are classified in \cite{Ni3}: apart from the abelian and the two-step ones,
there are only six others.
A characterization of Einstein nilradical with a simple Einstein derivation and
the classification of filiform Einstein nilradicals
(modulo known classification of filiform $\mathbb{N}$-graded Lie algebras) is given in \cite{Ni2}.

Our starting point is the following fact: if $\Phi$ is an Einstein derivation
of an Einstein nilradical $\n$, then for some constant $c < 0, \quad \Tr (\Phi \, \psi) = -c \, \Tr \psi$, for
any derivation $\psi$ of $\n$ (see Section~\ref{s:facts} for details). This motivates the following definition:

\begin{definition} \label{d:pE}
A derivation $\phi$ of a Lie algebra $\n$ is called \emph{pre-Einstein}, if it is semisimple, with all the eigenvalues
real, and
\begin{equation}\label{eq:pEtrace}
    \Tr (\phi \psi) = \Tr \psi,  \quad \text{for any $\psi \in \Der(\n)$}.
\end{equation}
\end{definition}

Here $\Der(\n)$ is the algebra of derivations of $\n$.
We call an endomorphism $A$ of a linear space real (nonnegative, positive), if all its eigenvalues are real
(respectively, nonnegative, positive). In the latter cases, we write $A \ge 0$ (respectively, $A>0$).
For any $\psi \in \Der(\n)$ we denote $\ad_\psi$ the corresponding inner derivation of $\Der(\n)$. If $\psi$ is
semisimple and real, the same is true for $\ad_\psi$.

Our main result is contained in Theorems~\ref{t:preE} and \ref{t:var}.

\begin{theorem} \label{t:preE}
{\ }
\begin{enumerate}[1.]
  \item
    \begin{enumerate}[(a)]
    \item Any Lie algebra $\g$ admits a pre-Einstein derivation $\phi_\g$.

    \item The derivation $\phi_\g$ is determined uniquely up to automorphism of $\g$.

    \item All the eigenvalues of $\phi_\g$ are rational numbers.
    \end{enumerate}

  \item Let $\n$ be a nilpotent Lie algebra, with $\phi$ a pre-Einstein derivation. If $\n$ is an Einstein
  nilradical, then its Einstein derivation is positively proportional to $\phi$ and
\begin{equation}\label{eq:pE>0}
  \phi > 0 \qquad \text{and} \qquad \ad_{\phi} \ge 0.
\end{equation}

\end{enumerate}
\end{theorem}

The inequalities \eqref{eq:pE>0} are necessary, but not sufficient to guarantee that a nilpotent Lie algebra is an
Einstein nilradical. Combining the idea of the pre-Einstein derivation with \cite[Theorem 6.15]{Heb}
we give a variational characterization of Einstein nilradicals, which
answers question~\eqref{q:A}. Denote $\mathcal{V} = \wedge^2 (\Rn)^* \otimes \Rn$ the space of skew-symmetric bilinear
maps on~$\Rn$. Let $\mu$ be an element of $\mathcal{V}$ defining a nilpotent Lie algebra $\n= (\Rn, \mu)$. Choose and
fix a pre-Einstein derivation $\phi$ of $\n$, and define the subalgebra $\g_\phi \subset \slg(\n)$ by
\begin{equation}\label{eq:subalgebra}
    \g_\phi = \z(\phi) \cap \Ker(t),
\end{equation}
where $\z(\phi)$ is the centralizer of $\phi$ in $\slg(n)$ and $t$ is a linear functional on $\slg(n)$ defined by
$t(A)=\Tr(A \, \phi)$. Let $G_\phi \subset \SL(n)$ be the connected Lie group with the Lie algebra
$\g_\phi$ (explicitly given by \eqref{eq:Ggphi}). Define the action of $G_\phi$ on the linear space $\mathcal{V}$
by $g.\nu(X,Y)=g\nu(g^{-1}X,g^{-1}Y)$ for $\nu \in \mathcal{V}, \, g \in G_\phi$.

Choose an arbitrary inner product $\ip$ on $\Rn$, with respect to which $\phi$ is symmetric, and define
$\| \nu \|^2= \sum_{i,j} \|\nu(E_i,E_j)\|^2$, where $\{E_i\}$ is an orthonormal basis for $\Rn$ and
$\nu \in \mathcal{V}$.

\begin{theorem} \label{t:var}
For a nilpotent Lie algebra $\n=(\Rn, \mu)$ with a pre-Einstein derivation $\phi$, the following conditions are
equivalent:
\begin{enumerate}[\rm(i)]
  \item $\n$ is an Einstein nilradical;
  \item the orbit $G_\phi.\mu \subset \mathcal{V}$ is closed;
  \item the function $f:G_\phi \to \Rn$ defined by $f(g)=\|g.\mu\|^2$ has a critical point.
\end{enumerate}
\end{theorem}

\begin{remark} \label{rem:critical}
Note that the equivalence of (i) and (ii) gives the characterization of Einstein nilradicals completely in terms of the
Lie algebra structure of $\n$, with no inner product involved.

The equivalence of (i) and (iii) is similar to the well-known criterion for
a homogeneous space of a \emph{unimodular} Lie group to be Einstein:
as it follows from \eqref{eq:riccinil} of Section~\ref{s:facts}, the scalar curvature of a metric nilpotent Lie algebra
is the squared norm of the Lie bracket times a constant. 

In (iii), one can replace ``has a critical point" by ``attains the minimum". Moreover, if $f$ attains the minimum
at $g \in G_\phi$ and $\mu_0=g.\mu$, then the metric Lie algebra $(\Rn, \mu_0, \ip)$ is isometrically
isomorphic to the nilradical of an Einstein metric solvable Lie algebra (see Theorem~\ref{t:varm} in
Section~\ref{s:thpE}).
\end{remark}

Note in conjunction with (ii), that the $\SL(n)$-orbit of a non-abelian nilpotent Lie algebra is never closed
\cite[Theorem 8.2]{La3}.
One can therefore either work with a smaller group (as in Theorem~\ref{t:var}), or consider a different function.
One such function naturally related to the group action is the squared norm of the moment map. As shown in \cite{La4,LW},
a nilpotent Lie algebra $\n=(\Rn, \mu)$ is an Einstein nilradical if and only if the squared
norm of the moment map (with respect to some inner product on $\Rn$) of the $\SL(n)$-action on $\mathcal{V}$
attains its minimum on the orbit of $\mu$.

As an application of Theorems~\ref{t:preE} and \ref{t:var}, we consider nilpotent Lie algebras having a nice basis.

\begin{definition}\label{d:nice}
Let $\{X_1, \ldots, X_n\}$ be a basis for a nilpotent Lie algebra $\n$, with
$[X_i, X_j]=\sum_k c_{ij}^k X_k$. The basis $\{X_i\}$ is called \emph{nice}, if
for every $i,j, \quad \#\{k: c_{ij}^k \ne 0\}$ $\le 1$, and for every $i,k, \quad \#\{j: c_{ij}^k \ne 0\} \le 1$.
\end{definition}

Although the condition of having a nice basis looks rather restrictive, the nilpotent Lie algebras with a nice basis
are not uncommon. Such algebras often appear in the classification lists, especially in the low-dimensional cases.
Every Lie algebra admitting a derivation with all the eigenvalues of multiplicity one has a nice basis (recall that to
be an Einstein nilradical, a nilpotent Lie algebra must admit a positive real semisimple derivation). In particular,
filiform algebras admitting an $\mathbb{N}$-gradation have a nice basis \cite{Ni2}.
Every two-step nilpotent algebra attached to a graph has a nice basis \cite{LW}.

Given a nilpotent algebra $\n$ of dimension $n$ with a nice basis, introduce the following objects.
In a Euclidean space $\Rn$ with the inner product $(\cdot, \cdot)$ and an orthonormal basis
$f_1, \ldots, f_n$, define the finite subset $\mathbf{F}=\{(Y)_{ij}^k= f_i+f_j-f_k: c_{ij}^k \ne 0, i < j\}$.
Denote $L$ the affine span of $\mathbf{F}$, the smallest affine subspace of $\Rn$ containing $\mathbf{F}$,
and $\Conv(\mathbf{F})$ the convex hull of $\mathbf{F}$.
Let $m = \#\mathbf{F}$. Fix an arbitrary ordering of the set $\mathbf{F}$ and define an $m \times n$ matrix $Y$
with the rows $(Y)_{ij}^k$. Namely, if for $1 \le a \le m$, the $a$-th element of $\mathbf{F}$ is $f_i+f_j-f_k$,
then the $a$-th row of $Y$ has $1$ in the columns $i$ and $j$, $-1$ in the column $k$, and zero elsewhere.
Denote $[1]_m$ an $m$-dimensional vector all of whose coordinates are ones.

We have the following theorem:

\begin{theorem}\label{t:nice}
A nonabelian nilpotent Lie algebra $\n$ with a nice basis is an Einstein nilradical if and only if any of the
following two equivalent conditions hold:

\begin{enumerate}[\rm(i)]
    \item the projection of the origin of $\Rn$ to $L$ lies in the interior of $\Conv(\mathbf{F})$.
    \item there exists a vector $\al \in \br^m$ with positive coordinates satisfying
    $YY^t \al = [1]_m$.
  \end{enumerate}
\end{theorem}

Note that by \cite[Theorem 1]{Pay}, a metric nilpotent Lie algebra is nilsoliton
if and only if the equation $YY^t \al = -2c \,[1]_m$ holds with respect to the basis of Ricci eigenvectors
(where the components of $\al$ are the squares of the structural constants).

As another application of Theorems~\ref{t:preE} and \ref{t:var}, we prove Theorem~\ref{t:rvsc}, which says that a
nilpotent Lie algebra, which is \emph{complex} isomorphic to an Einstein nilradical, is an Einstein nilradical by
itself, and Theorem~\ref{t:prod}, which shows
that the direct sum of nilpotent Lie algebras is an Einstein nilradical if and only if all the summands are Einstein
nilradicals (see Section~\ref{s:appl}).

In Section~\ref{s:2step}, we consider two-step nilpotent Einstein nilradicals. Informally, we prove that

\begin{quote}
\emph{a typical two-step nilpotent Lie algebra is an Einstein nilradical of eigenvalue type $(1,2; q,p)$}.
\end{quote}

There seem to be no commonly accepted notion of what a ``typical" nilpotent Lie algebra is \cite{Luk}.
We mean the following. A two-step nilpotent Lie algebra $\n$ is said to be of type $(p, q)$, if $\dim \n = p+q$ and
$\dim [\n, \n] = p$ (clearly, $1 \le p \le \frac12 q(q-1)$). Any such algebra is determined by a point in the
linear space $\mathcal{V}(p,q)=(\bigwedge^2 \br^q)^p$, with two points giving isomorphic algebras if and only if they
lie on the same orbit of the action of $\GL(q) \times \GL(p)$ on $\mathcal{V}(p,q)$. The space of isomorphism classes
of the algebras of type $(p,q)$ is a compact non-Hausdorff space, the quotient of an open and dense subset of
$\mathcal{V}(p,q)$ by the action of $\GL(q) \times \GL(p)$.

\begin{theorem}\label{t:twostepopen} Suppose $q \ge 6$ and $2 < p < \frac12 q(q-1)-2$, or $(p,q)=(5,5)$. Then
\begin{enumerate}[\rm (i)]
    \item
    there is a continuum isomorphism classes of two-step nilpotent Lie algebras of type $(p,q)$; each of them
    has an empty interior in the space $\mathcal{V}(p,q)$ \emph{\cite{Eb1}};
    \item
    the space $\mathcal{V}(p,q)$ contains an open and dense subset corresponding to two-step nilpotent Einstein
    nilradicals of eigenvalue type $(1,2; q,p)$
    \footnote{After this paper was written, the author became aware that the
result equivalent to Theorem~\ref{t:twostepopen} is independently proved in \cite{Eb3}.}.
\end{enumerate}
\end{theorem}

Two-step nilpotent Lie algebras of the types excluded by Theorem~\ref{t:twostepopen} can be completely
classified \cite{GT}.
Using Theorem~\ref{t:nice} we find all the two-step nilpotent Einstein nilradicals with $q \le 5$, $(p,q) \ne(5,5)$.

\medskip

The paper is organized as follows. Section~\ref{s:facts} gives the background on Einstein solvmanifolds.
In Section~\ref{s:thpE}, we prove Theorem~\ref{t:preE} and Theorem~\ref{t:varm}, which implies Theorem~\ref{t:var},
and also give a possible general strategy of determining whether a given nilpotent Lie algebra is an Einstein
nilradical. Section~\ref{s:appl} contains three applications of Theorem~\ref{t:preE} and Theorem~\ref{t:var}:
Theorem~\ref{t:nice} on Einstein nilradicals with a nice basis,
Theorem~\ref{t:rvsc}, which shows that the property of being an Einstein nilradical
is a property of the complexification of a real nilpotent Lie algebra, and Theorem~\ref{t:prod} on the
direct sum of Einstein nilradicals. In Section~\ref{s:2step}, we prove Theorem~\ref{t:twostepopen} and classify
low-dimensional two-step Einstein nilradicals (Proposition~\ref{p:lowdim}).

\section{Einstein solvmanifolds}
\label{s:facts}

Let $G$ be a Lie group with a left-invariant metric $Q$
obtained by the left translations from an inner product $\ip$ on the Lie algebra $\g$ of $G$.
Let $B$ be the Killing form of $\g$, and let $H \in \g$ be the \emph{mean curvature vector} defined by
$\<H, X\ra = \Tr \ad_X$.

The Ricci curvature $\mathrm{ric}$ of the metric Lie group $(G,Q)$ at the identity is given by
\begin{equation*}
    \mathrm{ric}(X)=-\<[H,X],X\ra-\frac12 B(X,X)-\frac12 \sum\nolimits_i \|[X,E_i]\|^2
    +\frac14 \sum\nolimits_{i,j} \<[E_i,E_j],X\ra^2,
\end{equation*}
for $X \in \g$, where $\{E_i\}$ is an orthonormal basis for $(\g, \ip)$ \cite[Eq.~(2.3)]{Heb}.

Equivalently, one can define
the Ricci operator $\ric$ of the metric Lie algebra $(\g, \ip)$ (the symmetric operator associated
to $\mathrm{ric}$) by
\begin{equation*}
\Tr \!\bigl((\ric + \frac12 (\ad_H + \ad_H^*) + \frac12 B )  A \bigr)
\!= \! \frac14 \sum_{i,j} \<A[E_i, E_j] - [AE_i, E_j] - [E_i, AE_j], [E_i, E_j]\ra,
\end{equation*}
for any $A \in \End(\g)$ (where $\ad_H^*$ is the metric adjoint of $\ad_H$).

If $(\n, \< \cdot, \cdot \ra)$ is a nilpotent metric Lie algebra, then $H = 0$ and $B = 0$, so
\begin{equation}\label{eq:riccinil}
\Tr (\ric_{\n} \, A) = \frac14 \sum\nolimits_{i,j} \<A[E_i, E_j] - [AE_i, E_j] - [E_i, AE_j], [E_i, E_j]\ra,
\end{equation}
for any $A \in \End(\n)$. Explicitly, for $X, Y \in \n$,
\begin{equation}\label{eq:riccinilexplicit}
\<\ric_{\n} X, Y \ra = \frac14 \sum_{i,j} \<X, [E_i, E_j]\ra \<Y, [E_i, E_j]\ra -
\frac12 \sum_{i,j} \<[X, E_i], E_j\ra \<[Y, E_i], E_j]\ra.
\end{equation}
By the result of \cite{La5}, any Einstein metric solvable Lie algebra is \emph{standard}, which means that
the orthogonal complement to the derived algebra $[\g, \g]$ is abelian.

It is proved in \cite{AK} that any Ricci-flat metric solvable Lie algebra is flat. By \cite{DM},
any Einstein metric solvable unimodular Lie algebra is also flat. We will therefore always assume $\g$
to be nonunimodular ($H \ne 0$), with an inner product of strictly negative scalar curvature $c \dim \g$.

Any Einstein metric solvable Lie algebra admits a rank-one reduction \cite[Theorem 4.18]{Heb}. This means that if
$(\g, \< \cdot, \cdot\ra)$ is such an algebra, with the nilradical $\n$ and the mean curvature vector $H$, then the
subalgebra $\g_1 = \mathbb{R}H \oplus \n$, with the induced inner product, is also Einstein. What is
more, the derivation $\Phi=\ad_{H|\n}:\n \to \n$ is symmetric with respect to the inner product, and all its
eigenvalues belong to $\gamma \mathbb{N}$ for some $\gamma > 0$. This implies, in particular, that the nilradical $\n$
of an Einstein metric solvable Lie algebra admits an $\mathbb{N}$-gradation defined by the eigenspaces of $\Phi$.
As proved in \cite[Theorem~3.7]{La1}, a necessary and sufficient condition for a metric nilpotent algebra
$(\n, \< \cdot, \cdot\ra)$ to be the nilradical of an Einstein metric solvable Lie algebra is
\begin{equation}\label{eq:ricn}
    \ric_\n = c \, \id_\n + \Phi, \quad \text{for some $\Phi \in \Der(\n)$},
\end{equation}
where $c \dim \g < 0$ is the scalar curvature of $(\g, \< \cdot, \cdot\ra)$. This equation, in fact, defines
$(\g, \< \cdot, \cdot\ra)$ in the following sense: given a metric nilpotent Lie algebra whose Ricci operator
satisfies \eqref{eq:ricn}, with some constant $c < 0$ and some $\Phi \in \Der(\n)$, one can define $\g$ as a
one-dimensional extension of $\n$ by $\Phi$. For such an extension $\g = \mathbb{R}H \oplus \n, \; \ad_{H|\n} = \Phi$,
and the inner product
defined by $\<H, \n \ra = 0,\; \|H\|^2 = \Tr \Phi$ (and coinciding with the existing one on $\n$) is Einstein, with
scalar curvature $c \dim \g$. A nilpotent Lie algebra $\n$ which admits an inner product
$\< \cdot, \cdot\ra$ and a derivation $\Phi$ satisfying \eqref{eq:ricn} is called an \emph{Einstein nilradical}, the
corresponding derivation $\Phi$ is called an \emph{Einstein derivation}, and the inner product $\< \cdot, \cdot\ra$
the \emph{nilsoliton metric}.

As proved in \cite[Theorem 3.5]{La1}, a nilpotent Lie algebra admits no more than one nilsoliton metric, up to
conjugation  by $\Aut(\n)$ and scaling (and hence, an Einstein derivation, if it exists, is unique, up to conjugation
and scaling). Equation \eqref{eq:ricn}, together with \eqref{eq:riccinil}, implies that if $\n$ is an Einstein
nilradical, with $\Phi$ the Einstein derivation, then for some $c <0$
\begin{equation}\label{eq:tracestandard}
    \Tr (\Phi \, \psi) = - c \, \Tr \psi,  \quad \text{for any $\psi \in \Der(\n)$}.
\end{equation}
The set of eigenvalues $\la_i$ and their multiplicities $d_i$ of the Einstein derivation $\Phi$ of an Einstein nilradical
$\n$ is called the \emph{eigenvalue type} of $\n$ (and of $\Phi$). The eigenvalue type is usually written as
$(\la_1, \ldots, \la_p \, ; \, d_1, \ldots, d_p)$ (note that the $\la_i$'s are defined up to positive multiple).

Throughout the paper, $\oplus$ means the direct sum of linear spaces (even when the summands are Lie algebras).
Any semisimple endomorphism $A$ of a linear space $V$ admits a decomposition into the real and the imaginary part:
$A=A^{\mathbb{R}} + A^{i\mathbb{R}}$. The operator
$A^{\mathbb{R}}$ is defined as follows: if $V_1, \ldots, V_m$ are the eigenspaces of $A$ acting on $V^\mathbb{C}$,
with eigenvalues $a_1, \ldots, a_m \in \bc$ respectively, then $A^{\mathbb{R}}$ acts by multiplication by
$b \in \mathbb{R}$ on every subspace $(\oplus_{k: \mathrm{Re} a_k = b} V_k) \cap V$. For any semisimple
$A \in \End(V)$, the operators $A, A^{\mathbb{R}}$, and $A^{i\mathbb{R}}$ commute. If $\psi$ is a semisimple derivation
of a Lie algebra $\g$, then both $\psi^{\mathbb{R}}$ and $\psi^{i\mathbb{R}}$ are also derivations.

\section{Pre-Einstein Derivation. Proof of Theorems~\ref{t:preE} and \ref{t:var}}
\label{s:thpE}

In this section, we prove Theorem~\ref{t:preE} and Theorem~\ref{t:varm}, which contains Theorem~\ref{t:var}, and
describe a possible general approach to answer questions \eqref{q:A} and \eqref{q:B} from Section~\ref{s:intro}.

\begin{proof}[Proof of Theorem~\ref{t:preE}]
1. (a) The algebra $\Der (\g)$ is algebraic. Let $\Der (\g) = \s \oplus \tg \oplus \n$ be its  Levi-Mal'cev
decomposition, where
$\tg \oplus \n$ is the radical of  $\Der (\g)$, $\s$ is semisimple, $\n$ is the set of all nilpotent elements in
$\tg \oplus \n$ (and is the nilradical of $\tg \oplus \n$), $\tg$ is a torus, an abelian subalgebra consisting of
semisimple elements, and $[\tg, \s] = 0$. With any $\psi \in \tg,\; \psi^{\mathbb{R}}$ and $\psi^{i\mathbb{R}}$ are also
in $\tg$. The subspaces
$\tg_c=\{\psi^{\mathbb{R}}\,:\, \psi \in \tg\}$ and $\tg_s=\{\psi^{i\mathbb{R}}\,:\, \psi \in \tg\}$ are the compact
and the fully $\mathbb{R}$-reducible tori (the elements of $\tg_s$ are diagonal matrices in some basis for
$\g$), $\tg_s \oplus \tg_c = \tg$.

The quadratic form $b$ defined on $\Der (\g)$ by $b(\psi_1, \psi_2) = \Tr (\psi_1  \psi_2)$ is invariant
($b(\psi_1, [\psi_2, \psi_3]) = b([\psi_1, \psi_3], \psi_2))$. In general, $b$ is degenerate, with
$\Ker b = \n$, 
so for any $\psi \in \n,\; b(\tg, \psi) = \Tr \psi = 0$. As $\s$ is semisimple and $[\tg, \s] = 0$, we also have
$b(\tg, \psi) = \Tr \psi = 0$, for any $\psi \in \s$. Moreover, for any
$\psi \in \tg_c, \quad b(\tg_s, \psi) = \Tr \psi = 0$.

Hence to find a pre-Einstein derivation for $\g$ it suffices to find an element $\phi \in \tg_s$ which satisfies
\eqref{eq:pEtrace}, for all $\psi \in \tg_s$. Such a $\phi$ indeed exists, as the restriction of $b$ to $\tg_s$ is
nondegenerate (even definite) and is unique, when a particular torus $\tg$ is chosen.

(b) The subalgebra $\s \oplus \tg$ is a maximal fully reducible subalgebra of $\Der(\g)$. As by \cite[Theorem 4.1]{Mos},
the maximal fully reducible subalgebras of $\Der(\g)$ are conjugate by an inner automorphism of $\Der(\g)$ (which
corresponds to an automorphism of $\g$), and then $\tg$, the center of $\s \oplus \tg$, is defined uniquely, the
uniqueness of $\phi$, up to automorphism, follows.

(c) The proof is similar to that of \cite[Theorem 4.14]{Heb}. Suppose $\phi$
has eigenvalues $\la_i$, with multiplicities $d_i$ respectively, $i=1, \ldots, p$. In a Euclidean space $\mathbb{R}^p$
with a fixed orthonormal basis $f_i$, consider all the vectors of the form $f_i+f_j-f_k$ such that
$\la_i+\la_j-\la_k = 0$.
In their linear span choose a basis $v_k, \; k = 1, \ldots , m$, consisting of vectors of the above form and
introduce a $p \times m$ matrix $F$ whose vector-columns are the $v_k$'s. Then any vector
$\nu = (\nu_1, \ldots, \nu_p)^t \in \mathbb{R}^p$ satisfying $F^t\nu =0$ defines a derivation $\psi = \psi(\nu)$ having
the same eigenspaces as $\phi$, but with the corresponding eigenvalues $\nu_i$. From \eqref{eq:pEtrace} we must have
$\sum d_i (\la_i -1) \nu_i = 0$, for any such $\nu$, which implies that the vector
$(d_1 (\la_1 -1), \ldots, d_p (\la_p -1))^t$ belongs to the column space of $F$. So there exists $x \in \mathbb{R}^m$
such that $\la = [1]_p + D^{-1}Fx$, where $\la = (\la_1, \ldots, \la_p)^t$, $[1]_p = (1, \ldots, 1)^t \in \mathbb{R}^p$,
and $D = \diag (d_1, \ldots, d_p)$. As $\phi$ by itself is a derivation, we have $F^t \la = 0$, which implies
$F^t[1]_p + F^tD^{-1}Fx = 0$, so that $x = - (F^tD^{-1}F)^{-1}[1]_m$, as $F^t[1]_p = [1]_m$ and $\rk \, F = m$. Then
$\la = [1]_p - D^{-1}F(F^tD^{-1}F)^{-1}[1]_m$ and the claim follows, as all the entries of $D$ and of $F$ are integers.

2. Suppose that $\n$ is an Einstein nilradical, with an Einstein derivation $\Phi$ and a nilsoliton inner product
$\ip$. Then $\Phi$ is semisimple, real and satisfies \eqref{eq:tracestandard}, with some negative
constant $c \in \br$, so $(-c)^{-1} \Phi$ is an Einstein derivation. Moreover, $\Phi > 0$ (as follows
from \cite[Theorem~4.14]{Heb}) and $\ad_{\Phi} \ge 0$. To prove the latter inequality, we use the fact that for
any $\psi \in \Der(\n), \quad \Tr (\Phi \, [\psi, \psi^*]) \ge 0$ (see assertion~2 of Lemma~2 of \cite{Ni1}).
If $\ad_{\Phi} \psi = \la \psi$ for some $\la \in \br$, then $\la \, \Tr(\psi \psi^*) \ge 0$.

\end{proof}

Note that the pre-Einstein derivation of a semisimple Lie algebra is trivial (zero), and it may well happen that
the pre-Einstein derivation of a nilpotent Lie algebra is zero (for instance, for a characteristically nilpotent
algebra).

Inequalities \eqref{eq:pE>0} from assertion 2 of Theorem~\ref{t:preE} can be used to prove that certain
nilpotent Lie algebras are not Einstein nilradicals. Of course, if $\n$ has no positive derivations at all,
it is not an Einstein nilradical. Example~1 of \cite{Ni1} shows that even when a nilpotent algebra has
positive derivations, its pre-Einstein derivation can be nonpositive. Nilpotent algebras whose
pre-Einstein derivation is positive, but the inequality $\ad_\phi \ge 0$ is violated, are more common.
One example is given below.

\begin{example}
Let $\n$ be a two-step nilpotent $12$-dimensional algebra of type $(2, 10)$ (see Section~\ref{s:2step})
defined by the relations
$[X_1,X_3]=[X_2,X_4]=[X_5,X_9]=[X_6,X_{10}]$ $=Z_1,\; [X_1,X_4]=[X_5,X_8]=[X_6,X_9]=[X_7,X_{10}]=Z_2$.
A direct computation (based, for instance, on \cite[Proposition 3.1]{Luk}) shows that the derivation
$\phi$ given by the diagonal matrix $\frac{1}{55}\diag(43,42,42,43,42,43,44,44,43,42,85,86)$ with respect to
the basis $\{X_i,Z_\al\}$ is pre-Einstein. An endomorphism $\psi$ sending $X_4$
to $X_2$ and all the other basis vectors to zero is a derivation satisfying $\ad_\phi\psi=-\frac{1}{55}\psi$.
\end{example}

In the proof of Theorem~\ref{t:var}, we combine the idea of the pre-Einstein derivation with \cite[Theorem~6.15]{Heb},
which in turn uses the results of \cite{RS} (see also \cite[Section 6]{La4}).
The group $G_\phi$ is explicitly constructed as follows.
Given a nilpotent algebra $\n$, fix a pre-Einstein derivation $\phi$. Let $\n_j$ be the eigenspaces of $\phi$,
with the corresponding eigenvalues $\la_j, \; j=1, \ldots, p$. Denote $a_j= N\la_j$, where
$N$ is the least common multiple of the denominators of the $\la_j$'s. Then $G_\phi$ is the identity component of
the subgroup $\tilde{G}_\phi \subset \prod_{j=1}^p \GL(\n_j) \subset \GL(n)$ defined by
\begin{equation}\label{eq:tildegphi}
\tilde{G}_\phi=\{(g_1, \ldots, g_p) \, : \, g_j \in \GL(\n_j), \,
\prod\nolimits_{j=1}^p \det g_j = \prod\nolimits_{j=1}^p (\det g_j)^{a_j} = 1\}.
\end{equation}
As all the $a_j$ are integers (although some could be zero or negative), the group $\tilde{G}_\phi$ is real algebraic.
The group $G_\phi$ is reductive, with the Lie algebra $\g_\phi$, and is given explicitly by
\begin{equation}\label{eq:Ggphi}
G_\phi=\tilde{G}_\phi \cap \prod\nolimits_{j=1}^p \GL^+(\n_j), \quad
\text{where }\GL^+(V)= \{g \in \GL(V) : \det g > 0 \}.
\end{equation}

We start with the following technical lemma (cf. \cite[Section~6.4]{Heb}).

\begin{lemma} \label{l:tech}
\begin{enumerate}[\rm 1.]
  \item
  Let $\n=(\Rn, \mu)$ be a nilpotent Lie algebra with a pre-Einstein derivation $\phi$. Then $\phi$ is a pre-Einstein
  derivation for every algebra $(\Rn, g.\mu), \; g \in G_\phi$. The group $G_\phi$ and the orbit $G_\phi.(g.\mu)$ are
  the same, for all $g \in G_\phi$.

  \item \label{it:bltech}
  Let $\n=(\Rn, \mu)$ be a nilpotent Lie algebra and let $\psi$ be a semisimple derivation of $\n$ with rational 
  eigenvalues (not necessarily a pre-Einstein). Let $\ip$ be an inner product on $\n$, with respect to which $\psi$ is
  symmetric. Define the algebra $\g_\psi$ and the group $G_\psi$ for $\psi$ by \eqref{eq:subalgebra}
  and \eqref{eq:Ggphi} respectively and the function $f: G_\psi \to \br$ by $f(g)=\|g.\mu\|^2$. Then:
  \begin{enumerate}[\rm(a)]
    \item For every $g \in G_\psi$ and $A \in \End(\n), \quad
    \Tr(\ric_{g.\mu} \, A)=\frac12 {\frac{d}{dt}}_{|t=0}f(\exp(tA)g)$.
    \item Let $Q$ be an inner product on $\End(\n)$ defined by $Q(A_1, A_2)=\Tr(A_1 \, A_2^*)$ (adjoint with
    respect to $\ip$), and $\mathfrak{Z}(\psi)$ be the centralizer of $\psi$ in $\End(\n)$. Then for
    every $\g \in G_\psi$ and $A \perp \mathfrak{Z}(\psi)$, ${\frac{d}{dt}}_{|t=0}f(\exp(tA)g)=0$.
  \end{enumerate}

  \item
  In the settings of \ref{it:bltech}, suppose that the function $f$ has a critical point $g_0 \in G_\psi$. Then:
\begin{enumerate}[\rm(a)]
  \item $\n$ is an Einstein nilradical, with $\ip$ a nilsoliton metric for $(\Rn,g_0.\mu)$.
  \item if $\n$ is nonabelian, then $\psi$ is proportional to a pre-Einstein derivation of $\n$.
\end{enumerate}
\end{enumerate}
\end{lemma}

\begin{proof}
1. It is easy to see that $\Der(g.\mu)=g\Der(\mu)g^{-1}$, for any $g \in \GL(n)$. In particular, if $g \in G_\phi$ (and
hence commutes with $\phi$), the endomorphism $\phi$ is a real semisimple derivation of $(\Rn,g.\mu)$ satisfying
\eqref{eq:pEtrace} with any $\psi \in \Der(g.\mu)$.

2. Assertion (a) follows from \eqref{eq:riccinil} (with the Lie bracket $[\cdot,\cdot]=g.\mu$). Assertion (b) follows
from (a) and the fact that the Ricci tensor commutes with every symmetric derivation \cite[Lemma~2.2]{Heb}, in
particular, with $\psi$.

3. From the assumption and assertion \ref{it:bltech} we obtain that $\Tr(\ric_{g_0.\mu} \, A) = 0$, for all
$A \in \g_\psi \oplus (\mathfrak{Z}(\psi))^\perp$. It follows from \eqref{eq:subalgebra} that
$\ric_{g_0.\mu} \in \Span(\id, \psi)$, that is, equation \eqref{eq:ricn} holds, with some $c \in \br$ and some
$\psi \in \Der(g_0.\mu)$. Then $\ip$ is a nilsoliton metric for $(\Rn, g_0.\mu)$ by \cite[Theorem~3.7]{La1}.
The nilsoliton inner product for $\n=(\Rn,\mu)$ can be then taken as $\<X,Y\ra'= \<g_0^{-1}X,g_0^{-1}Y\ra$, for
$X,Y \in \Rn$. This proves (a). To prove (b), note that $\ric_{g_0.\mu} = c\,\id + a\psi$ for some $c, a \in \br$,
as shown above, and $\ric_{g_0.\mu} = c_1 \id + \Phi$, for some $c_1 \in \br$, where $\Phi$ is the Einstein
derivation of the metric Lie algebra $(\Rn,g_0.\mu,\ip)$. By assertion 2 of Theorem~\ref{t:preE}, $\Phi = a_1 \phi$
for some $a_1 \in \br$, where $\phi$ is a pre-Einstein derivation for $(\Rn,g_0.\mu)$, and hence for $\n$ by
assertion 1. Then $c\,\id + a\psi=c_1\id + a_1\phi$. If $c \ne c_1$, then $\id \in \Der(g_0.\mu)$,
so $\n$ is abelian. Otherwise, $a\psi=a_1\phi$. If $a=0$, then $\ric_{g_0.\mu} = c\,\id$, and $\n$ is again
abelian by \cite[Theorem~2.4]{Mil}.
\end{proof}

Using the results of \cite{RS} and Lemma~\ref{l:tech} we prove the following theorem, which contains
Theorem~\ref{t:var}.

\begin{theorem} \label{t:varm}
Let $\n=(\Rn, \mu)$ be a nilpotent Lie algebra with a pre-Einstein derivation $\phi$. Let $\ip$ be an inner product
on $\n$, with respect to which $\phi$ is symmetric. Define the algebra $\g_\phi$ and the group $G_\phi$
by \eqref{eq:subalgebra} and \eqref{eq:Ggphi} respectively and the function $f: G_\phi \to \br$ by $f(g)=\|g.\mu\|^2$.
\begin{enumerate}[\rm 1.]
  \item
  The following conditions are equivalent:
  \begin{enumerate}[\rm(i)]
    \item $\n$ is an Einstein nilradical;
    \item the orbit $G_\phi.\mu \subset \mathcal{V}$ is closed;
    \item the function $f:G_\phi \to \Rn$ has a critical point;
    \item the function $f:G_\phi \to \Rn$ attains its minimum.
  \end{enumerate}
  If $g_0 \in G_\phi$ is a critical point for $f$, then $\<g_0^{-1}\cdot,g_0^{-1}\cdot\ra$ is a nilsoliton inner product
  for $\n$.

  \item
  Suppose the orbit $G_\phi.\mu$ is not closed. Then there exists a unique $\mu_0 \in \overline{G_\phi.\mu}$ such that
  the orbit $G_\phi.\mu_0$ is closed. The following holds:
\begin{enumerate}[\rm(a)]
  \item
  the algebra $\n_0=(\Rn,\mu_0)$ is an Einstein nilradical not isomorphic to $\n$;
  \item
  either $\phi$ is a pre-Einstein derivation for $\n_0$, or $\n_0$ is abelian;
  \item
  there exists $A \in \g_\phi$ such that $\mu_0=\lim_{t\to\infty} \exp(tA).\mu$;
  \item
  such an $A$ can be chosen symmetric, with integer eigenvalues.
\end{enumerate}

\end{enumerate}
\end{theorem}

\begin{proof}
1. The group $G_\phi$ is a real reductive algebraic group with the Lie algebra $\g_\phi$, which
has a Cartan decomposition compatible with $\ip$ (on the symmetric and the skew-symmetric endomorphisms).
The equivalence of (ii) and (iv) follows directly from \cite[Theorem~4.4]{RS}. The equivalence of (ii) and (iii)
follows from \cite[Theorem~4.3]{RS} and the fact that the group $G_\phi$ is the same for all the points of the orbit
$G_\phi.\mu$ (assertion 1 of Lemma~\ref{l:tech}). The implication (iii)$\Rightarrow$(i) and the fact that the nilsoliton
inner product has the required form follow from assertion 3 of Lemma~\ref{l:tech}.

To prove the converse implication (i)$\Rightarrow$(iii),
suppose that $\n$ is an Einstein nilradical. Then for some $g \in \GL(n)$, the inner
product $\ip$ is nilsoliton on $(\Rn, g.\mu)$. We
want to show that $g$ can be chosen from the group $G_\phi$.
By \eqref{eq:ricn} and assertion 2 of Theorem~\ref{t:preE}, for any nilsoliton inner product there exists
a symmetric pre-Einstein derivation, so by conjugation  by an element of $\Aut(\n)$ we can assume that there exists a
nilsoliton inner product  $\ip'$, with respect to which $\phi$ is symmetric.
It follows that  for all $X, Y \in \Rn, \quad \<X, Y\ra' = \<hX, Y\ra$ for some positive definite $h$ symmetric with
respect to $\ip$ and belonging to $H$, the centralizer of $\exp\phi$ in $\GL^+(n)$.
Then there exists $h_1 \in H$, with all the eigenvalues positive, such that $h_1^*h_1 =h$ ($h_1^*$ is the adjoint with
respect to $\ip$). Hence $\ip$ is a nilsoliton inner product on the Lie algebra
$(\Rn, h_1^{-1}.\mu)$. The two-dimensional abelian Lie group $H_0=\exp(\Span(\id, \phi))$ lies in the center of
$H$, and $H=H_0 G_\phi$, so $h_1=gg_0$ for some $g \in G_\phi$ and $g_0 \in H_0$ (in fact, $g$ even belongs to
$\exp(\g_\phi)$). Moreover, $g_0= e^t \exp(s \phi)$ for some $t,s \in \br$. As $\exp(s \phi)$ is an automorphism of
$\mu$ and $e^t$ is a scaling, the inner product $\ip$ is nilsoliton on $(\Rn, g^{-1}.\mu)$,
with $g \in G_\phi$, which implies that $g^{-1}$ is a critical point of the function $f$.

2. The fact that the closure of $G_\phi.\mu$ contains a unique closed orbit $G_\phi.\mu_0$ follows from \cite[9.3]{RS}.
By \cite[Lemma~3.3]{RS}, there exists $A \in \g_\phi$ such that $\mu_0=\lim_{t\to\infty} \exp(tA).\mu$. Such an $A$
can be chosen symmetric with respect to $\ip$ and also such that the subgroup $\exp(tA) \in G_\phi$ is algebraic,
which implies that the eigenvalues of some nonzero multiple of $A$ are integers. This proves (c) and (d).

To prove (a) and (b), note that $\mu_0$, being a limit of nilpotent Lie brackets, is a nilpotent Lie bracket itself.
Regardless of whether $\phi$ is or is not a pre-Einstein derivation for the nilpotent Lie algebra $\n_0=(\Rn,\mu_0)$,
it follows from \cite[Theorems~4.3,4.4]{RS} that the function $\tilde f: G_\phi \to \br$ defined by
$\tilde f (g)= \|g.\mu_0\|^2$ has a critical point. Then by assertion 3(a) of Lemma~\ref{l:tech}, the algebra $\n_0$ is
an Einstein nilradical. It is not isomorphic to $\n$, as otherwise $\n$ were an Einstein nilradical and its
$G_\phi$-orbit would be closed by 1(ii) above. This proves (a).

By assertion 3(b) of Lemma~\ref{l:tech}, either $\n_0$ is abelian, or $\phi=a \phi_0$, where $\phi_0$ is a
pre-Einstein derivation for $\n_0$. As $\Tr \phi^2=\Tr \phi$
and $\Tr \phi_0^2=\Tr \phi_0$ by \eqref{eq:pEtrace}, it follows that either $\phi=\phi_0$, or $\Tr \phi_0=0$,
or $\phi=0$. The equation $\Tr \phi_0=0$ is not possible, as $\n_0$ is an Einstein nilradical, so $\phi_0 > 0$
(by 2 of Theorem~\ref{t:preE}).
If $\phi=0$, then $G_\phi=\SL(n)$, so $\n_0$ is abelian by \cite[Theorem 8.2]{La3} and the uniqueness of $\mu_0$.

\end{proof}

\begin{remark}
One possible reason why in 2(a) the limiting algebra $\n_0$ can be abelian is that there is no Einstein nilradical whose
pre-Einstein derivation is $\phi$. This could happen, for instance, when $\phi$ is nonpositive (assertion 2 of
Theorem~\ref{t:preE}), or when $\phi >0$, but the condition of \cite[Lemma~1]{Ni3} is violated. The latter case is
illustrated by the majority of free nilpotent Lie algebras.
The pre-Einstein derivation $\phi$ of the $p$-step nilpotent free Lie
algebra $\f(m,p)$ on $m$ generators is a positive multiple of the derivation having eigenvalues $1, 2, \ldots, p$, whose
eigenspaces are the homogeneous components of $\f(m,p)$. Although both inequalities \eqref{eq:pE>0} are satisfied,
$\f(m,p)$ is not an Einstein nilradical when say $m >2$ and $p>3$,
as no Einstein nilradical can have such
pre-Einstein derivation \cite[Lemma~2]{Ni3}. So for all $\n=\f(m,p)$, which are not Einstein nilradicals, the
algebra $\n_0$ is abelian.

We emphasize that the property of a nilpotent algebra $\n=(\Rn, \mu)$ to be an Einstein nilradical depends on the
closedness of the orbit of $\mu$ by the action of the group $G_\phi$ \emph{depending on $\n$}.
As the result, the closure of a nonclosed orbit $G_\phi.\mu \subset \mathcal{V}$ may contain
more than one Einstein nilradical.
For each of them, the orbit of the Lie bracket under the action of its own group is closed, but the $G_\phi$-orbit
is closed only for one of them (assertion 2 of Theorem~\ref{t:varm}).
{
\begin{example}
Let $\n'=(\br^{m},\mu')$ be a characteristically nilpotent Lie algebra. Its pre-Einstein derivation $\phi'$ is zero,
$G_{\phi'} = \SL(m)$, and the closure of the orbit $G_{\phi'}.\mu'$ contains the abelian algebra
$\br^{m}$. Let a nilpotent algebra $\n=(\br^{m+3}, \mu)$ be the direct sum of $\n'$ and the three-dimensional
Heisenberg algebra $\h_3$ given by $[X_1,X_2]=X_3$. The pre-Einstein derivation for $\h_3$ is
$\phi_2=\frac23 \diag(1,1,2)$. By
assertion 1 of Theorem~\ref{t:prod} below, the derivation $\phi=0_{\n'} \oplus \phi_2$ is a
pre-Einstein derivation for $\n$. As $\phi$
is not positive, the closure of $G_\phi.\mu$ contains the abelian algebra $\br^{m+3}$. On the other
hand, that closure also contains a nonabelian Einstein nilradical $\hat \n=(\br^{m+3}, \hat\mu)$, the direct sum
of $\br^m$ and $\h_3$. The derivation $\hat \phi=\id_{\n'} \oplus \phi_2$ is pre-Einstein for $\hat \n$,
the $G_{\hat\phi}$-orbit of $\hat\mu$ is closed, but the $G_\phi$-orbit is not.
\end{example}
}
\end{remark}

\begin{remark}
As follows from \cite[Theorem~4.3]{RS}, if $g_0 \in G_\phi$ is a critical point of the function $f$, then the set of
minimal brackets (the brackets with the smallest norm) in the orbit $G_\phi.\mu$ is $U_\phi.g_0.\mu$, where
$U_\phi=G_\phi \cap \SO(n)$. This implies that the set of critical points of $f$ is $U_\phi g_0 (\Aut(\n) \cap G_\phi)$.
\end{remark}

Theorems~\ref{t:preE} and \ref{t:var} suggest the following strategy for deciding whether a given nilpotent Lie algebra
$\n$ is an Einstein nilradical.

\begin{enumerate}[(a)]
  \item \label{it:a}
  Find a pre-Einstein derivation $\phi$ for $\n$. In practice, this can be done just by solving a system of linear
  equations. As it follows from the proof of assertion 1(a) of Theorem~\ref{t:preE}, every maximal torus of
  $\Der(\n)$ contains a pre-Einstein derivation, so one can choose a particular maximal torus to reduce the amount of
  calculations. Also note that if $\phi$ is a candidate for being a pre-Einstein derivation, it suffices to check
  the validity of \eqref{eq:pEtrace} only for semisimple derivations $\psi$ commuting with $\phi$.

  \item \label{it:b}
  If inequalities \eqref{eq:pE>0} are not satisfied, then $\n$ is not an Einstein nilradical. Otherwise, compute the
  algebra $\g_\phi$ and use Theorem~\ref{t:var} or  Theorem~\ref{t:varm}.

  \item \label{it:c}
  Sometimes instead of (\ref{it:b}) it is easier to prove that $\n$ is an Einstein nilradical by explicitly producing
  a nilsoliton inner product. The pre-Einstein derivation gives a substantial amount of information for finding it.
  First of
  all, by assertion 2 of Theorem~\ref{t:preE}, if a nilsoliton inner product exists, it can be chosen in such a way
  that the eigenspaces of $\phi$ are orthogonal. Secondly, the pre-Einstein derivation completely determines the
  eigenvalue type. Also, as by \eqref{eq:riccinil} $\Tr (\ric_\n \phi) = 0$, it follows
  from \eqref{eq:ricn} and Theorem~\ref{t:preE} that an inner product on $\n$ is nilsoliton if and only if
  $\ric_\n = c \, (\id-\phi)$ for some $c < 0$, where the expression for $\ric_\n$ is given by the right-hand side of
  \eqref{eq:riccinilexplicit}.
\end{enumerate}

\section{Applications of Theorems~\ref{t:preE} and \ref{t:var}}
\label{s:appl}

In this section, we use Theorem~\ref{t:preE} and Theorem~\ref{t:var} to prove the following three facts:
Theorem~\ref{t:nice}, which gives an easy-to-check condition for a nilpotent Lie algebra with a nice basis to be an
Einstein nilradical, Theorem~\ref{t:rvsc}, which says that the property of being an Einstein nilradical
is, in fact, a property of the complexification of a real nilpotent Lie algebra, and Theorem~\ref{t:prod}, which says
that the direct sum of nilpotent Lie algebras is an Einstein nilradical if and only if the summands are.

\begin{proof}[Proof of Theorem~\ref{t:nice}]
In brief, the proof goes as follows. We fix a nice basis and compute the pre-Einstein derivation and the group $G_\phi$.
Then we show that the closedness of the orbit of the diagonal subgroup of $G_\phi$ is completely controlled by the
convexity condition (i), and that if there is a critical point somewhere on $G_\phi$, then there is one on the diagonal.

Let $\mathcal{B}=(X_1, \ldots, X_n)$ be a nice basis for a nonabelian nilpotent Lie algebra $\n=(\Rn,\mu)$.
Let $\Lambda=\{(i,j,k) \, : \, i < j,\, c_{ij}^k \ne 0 \},\; \#\Lambda = m >0$. In a Euclidean space $\Rn$ with the inner
product $(\cdot, \cdot)$ and an orthonormal basis $f_1, \ldots, f_n$, define the finite subset
$\mathbf{F}=\{Y_a= f_i+f_j-f_k \, : \, a=(i,j,k) \in \Lambda\}$ and denote by $L$ the affine span of $\mathbf{F}$,
the smallest affine subspace of $\Rn$ containing $\mathbf{F}$.
Fix an arbitrary ordering of the set $\Lambda$ and define an $m \times n$ matrix $Y$
whose $a$-th row has $1$ in the columns $i$ and $j$, $-1$ in the column $k$, and zero elsewhere, where
$a=(i,j,k) \in \Lambda$.

The proof follows the steps at the end of Section~\ref{s:thpE}.

Fix the nice basis $\mathcal{B}$ and an inner product $\ip$ on $\n$, with respect to which the
nice basis is orthonormal. We say that an operator (an inner product) is diagonal, if its matrix with
respect to $\mathcal{B}$ is diagonal. For a vector $v=(v_1, \ldots, v_n)^t \in \Rn$, denote $v^D$ the
diagonal operator defined by $v^D X_i := v_i X_i$. Note that $v^D \in \Der(\n)$ if and only if $v \in \Ker Y$.
Similarly, for a vector
$r=(r_1, \ldots, r_n)^t \in \Rn$, with $r_i \ne 0$ denote $\ip_r$ the diagonal inner product defined
by $\<X_i, X_j\ra_r:= r_i^2 \delta_{ij}$.

As it follows from \eqref{eq:riccinilexplicit}, the Ricci operator $\ric_r$ of the metric algebra
$(\n, \ip_r)$ is diagonal and
\begin{equation}\label{eq:ricdiag}
\ric_r = (-\tfrac{1}{2} Y^t \beta)^D
\end{equation}
for a vector $\beta \in \br^m$ with
$\beta_a=(c_{ij}^kr_i^{-1}r_j^{-1}r_k)^2$, where $a=(i,j,k) \in \Lambda$.

We want to show that $\n$ admits a diagonal pre-Einstein derivation (note that $\n$ may have more than one nice basis).
Let $\phi$ be a diagonal derivation satisfying \eqref{eq:pEtrace} for all diagonal derivations $\psi$. Then
$\phi=v^D$, with $v \in \Ker Y$, and $(v,u)=([1]_n,u)$, for all $u \in \Ker Y$.
It follows that $v=[1]_n+Y^t\gamma$ for some $\gamma \in \br^m$, and $0=Yv=[1]_m + YY^t\gamma$,
so $\gamma=-\alpha$. Hence
\begin{equation}\label{eq:preEdiag}
\phi=\id-(Y^t\alpha)^D.
\end{equation}
To show that such a $\phi$ is indeed a pre-Einstein derivation, consider an arbitrary $\psi \in \Der(\n)$. By
\eqref{eq:riccinil}, $\Tr(\ric \; \psi)=0$, for any inner product on $\n$. In particular, for any diagonal
inner product $\ip_r$, equation \eqref{eq:ricdiag} implies that
$\Tr((Y^t \beta)^D \psi)=0$. Choose $v \in \Rn$ in such a way that $v^D$ and $\psi$ have the same diagonal entries. Then
$0=\Tr((Y^t \beta)^D v^D) = (\beta, Yv)$. This holds for any vector $\beta \in \br^m$ with the components
$\beta_a=(c_{ij}^kr_i^{-1}r_j^{-1}r_k)^2$, where $a=(i,j,k) \in \Lambda$.
Choose one such $a=(i,j,k)$ and take $r_i=r_j=e^{-t}$, and $r_l=1$ for $l \ne i,j$.
Substituting this to $(\beta, Yv)=0$, dividing by $e^{4t}$ and taking the limit when $t \to \infty$ we obtain
$(Yv)_a=0$. It follows that $Yv=0$, so that $v^D$, the
projection of $\psi$ to the diagonal, is also a derivation. Then $\Tr (\phi \psi)= \Tr (\phi v^D)= \Tr v^D= \Tr \psi$,
so the derivation $\phi$ given by \eqref{eq:preEdiag} (which is clearly semisimple and real) is indeed a pre-Einstein
derivation.

As in Theorem~\ref{t:var}, let $\g_\phi$ be the Lie algebra defined by \eqref{eq:subalgebra}, $G_\phi$ be the
Lie group defined by \eqref{eq:Ggphi}, and $f: G_\phi \to \br$ be the function defined by
$f(g)=\|g.\mu\|^2$, where $\mu$ is the Lie bracket of $\n$.

We first consider the behavior of $f$ restricted to the diagonal subgroup of $G_\phi$. Let $\g_D$ be the
abelian subalgebra consisting of all the diagonal elements of $\g_\phi$.

{
\begin{lemma}\label{l:nicediag}
1. The following three conditions are equivalent:
\begin{enumerate}[\rm (i)]
    \item
    the projection of the origin of $\Rn$ to $L$ lies in the interior of $\Conv(\mathbf{F})$.
    \item
    there exists a vector $\al \in \br^m$ with positive coordinates satisfying $YY^t \al = [1]_m$.
    \item
    the function $f_D:\g_D \to \br$ defined by $f_D(A)=f(\exp(A))$ for $A \in \g_D$ has a critical point.
\end{enumerate}

2. If $A \in \g_D$ is a critical point of $f_D:\g_D \to \br$, then $\exp(A)$ is a critical point of $f:G_\phi \to \br$,
so $\n$ is an Einstein nilradical by Theorem~\ref{t:var}.
\end{lemma}
\begin{proof}
1. First we show that the geometric condition (i) and the algebraic condition (ii) are equivalent.
Since $Y[1]_n=[1]_m$, all the points of $\mathbf{F}$ lie on the hyperplane $\{\sum_i x_i = 1\} \subset \Rn$. It
follows that $P$, the projection of the origin to $L$, is nonzero and $YP=\|P\|^2 [1]_m$. An easy computation shows
that $P=\|P\|^2 Y^t\alpha$, where $\alpha \in \br^m$ is an arbitrary vector satisfying the equation $YY^t \al = [1]_m$
(in particular, such an $\al$ exists).
For any such $\alpha, \quad (\al,[1]_m)= \|P\|^{-2}>0$, so $P=(\al,[1]_m)^{-1} Y^t\alpha$. Therefore,
if a vector $\al$ satisfying $YY^t \al = [1]_m$ can be chosen in such a way that all its coordinates are positive, then
$P$ lies in the interior of $\Conv(\mathbf{F})$. Conversely, suppose that $P$ lies in the interior of
$\Conv(\mathbf{F})$ and denote $C=\{a \, :\, Y_a \text{ is a corner point of } \Conv(\mathbf{F})\}$. Then
$P=Y^t \beta$ for some $\beta \in \br^m$, with $(\beta,[1]_m)=1$, such that $\beta_a>0$, if $a \in C$, and
$\beta_a=0$ otherwise.
For every $b \notin C$, let $Y_b=\sum_{a \in C} \la_{ab} Y_a$, with $\la_{ab} \ge 0$ and $\sum_{a \in C} \la_{ab}=1$.
Then, for a small $\varepsilon > 0$, $\sum_{a=1}^m \beta_a Y_a =
\sum_{a \in C} (\beta_a - \sum_{b \notin C} \varepsilon \la_{ab}) Y_a + \sum_{b \notin C} \varepsilon Y_b
=\sum_{a=1}^m \beta_a' Y_a$, with all the $\beta_a'$ positive and $(\beta',[1]_m)=1$. So $P=Y^t\beta'$, and
$\al=\|P\|^{-2} \beta'$ is the required vector, as $YP=\|P\|^2 [1]_m$.

To show that (i) and (ii) are equivalent to (iii), we start with the following observation: let $\{w_a\}$
be a finite set of vectors in a Euclidean space $V$, with $\Span(w_a)= V' \subset V$, and let $b_a$ be
positive numbers. Then the function $e: V \to \br$ defined by
$e(x)=\sum_a b_a e^{(w_a,x)}$ has a critical point if and only if it attains its minimum if and only if the origin of
$V$ lies in the interior (relative to $V'$) of the convex hull of the vectors $w_a$. Since $e(x+y)=e(x)$ for
$y \perp V'$, we lose no generality by replacing $V$ by $V'$ and $e$ by $e_{|V'}$. The fact that a critical point,
if it exists, is a minimum, is clear, as $e(x)$ is positive and convex. Next, at a critical point
$x, \quad de(x)=\sum_a b_a e^{(w_a,x)} w_a=0$, so $0$ lies in the interior of the convex hull of the $w_a$'s.
Conversely, if $M$ is a large positive number, the preimage $e^{-1}(0,M]$ is closed, nonempty, and is contained in the
convex polyhedron $\cap_a \{x \, : \: (w_a,x) \le \ln(Mb_a^{-1})\}$. If the origin lies in the interior of the convex
hull of the $w_a$'s, that polyhedron is bounded, so $e^{-1}(0,M]$ is compact, which proves the claim.

By \eqref{eq:preEdiag}, $\g_D=\{v^D \, : \, (v, [1]_n)=(v, Y^t\alpha)=0 \}$.
Then we obtain $f_D(v^D)=\|\exp(v^D).\mu\|^2=$ $ \sum_{a=(i,j,k) \in \Lambda} (c_{ij}^k)^2 \exp(2(v_k-v_i-v_j))
= \sum_{a=(i,j,k) \in \Lambda} (c_{ij}^k)^2 \exp(-2(Yv)_a)$. By the observation from the previous paragraph,
with $V=\{v :(v, [1]_n)=(v, Y^t\alpha)=0 \}$ and $w_a=Y_a-P$, the projections of the
vectors $Y_a$ to $V$, we find that $f_D$ has a critical point if and only if the origin lies in the interior of the
convex hull of the vectors $Y_a-P$, that is, if and only if $P$ lies in the interior of $\Conv(\mathbf{F})$.

2. Let $v^D \in \g_D$ be a critical point of $f_D$ and let $\mu'=\exp(v^D).\mu$. The basis $\mathcal{B}$ is still
nice for the Lie algebra $\n' =(\Rn, \mu')$, so by \eqref{eq:ricdiag} the Ricci operator $\ric_{\n'}$ of the
metric Lie algebra $(\n', \ip)$ is diagonal. By 2(a) of Lemma~\ref{l:tech},
$\frac{d}{dt}_{|t=0} f(\exp(tA).\exp(v^D))=2 \Tr (\ric_{\n'} A)$, for any $A \in \End(\n)$. The expression on the
right-hand side vanishes for all the $A$'s with the zero diagonal, as $\ric_{\n'}$ is diagonal, and for all
$A \in \g_D$, as $v^D$ is critical for $f_D$. Hence $\exp(v^D)$ is a critical point of $f$.
\end{proof}
}

Lemma~\ref{l:nicediag} proves the ``if" part of the theorem. To prove the ``only if" part, we will show that if the
function $f$ has a critical point somewhere on $G_\phi$, then it has a critical point on $\exp \g_D$.

We will need the following lemma, which slightly refines the Cartan
decomposition of the reductive real algebraic group $G_\phi$.
On the Lie algebra level, we have $\g_\phi=\mathfrak{u}_\phi \oplus \p=\mathfrak{u}_\phi \oplus \p_0 \oplus \g_D$,
where $\mathfrak{u}_\phi$ is the subalgebra of all the
skew-symmetric matrices from $\g_\phi$, $\p$ is the subspace of all the symmetric matrices from $\g_\phi$,
and $\p_0$ is the subspace of all the matrices with the zero diagonal from $\p$. On the Lie group level, let
$U_\phi = \mathrm{O}(n) \cap G_\phi$, a maximal compact subgroup of $G_\phi$, $G_D=\exp\g_D$, a maximal torus in
$G_\phi$, and let $P_0= \exp \p_0$.

{
\begin{lemma}\label{l:polar}
The map $U_\phi \times P_0 \times G_D \to G_\phi$ defined by $(u,p,g_D) \to upg_D$ for
$u \in U_\phi, \, p \in P_0, \, g_D \in G_D$, is surjective.
\end{lemma}
\begin{proof}[Proof of Lemma~\ref{l:polar}]
We follow the approach of \cite{WZ}.
Let $\pi(g)=g^t g$ for $g \in G_\phi$, and let $M=\pi(G_\phi)$.
Clearly, $\pi(g_1) = \pi(g_2)$ if and only if $g_2=ug_1$ for some $u \in U_\phi$.
The action of $G_\phi$ on itself from the right defines the action on the homogeneous space $M=G_\phi/U_\phi$ by
$x \to g^txg$ for $x \in M$. An inner product $Q$ on $\g_\phi$ defined by $Q(A_1,A_2)=\Tr(A_1A_2^t)$ is
$\mathrm{Ad}(U_\phi)$-invariant and gives, by right translations, a bi-invariant Riemannian metric on $M$.
The space $M$ with that metric is an Hadamard symmetric space with the de Rham decomposition
$\prod_{j=1}^p \SL(d_j)/\SO(d_j) \times \br^{p-2}$, where $d_1, \ldots, d_p$ are the multiplicities of the eigenvalues
of the pre-Einstein derivation $\phi$. The space $\g_D$ is a maximal abelian subalgebra in $\p=\mathfrak{u}_\phi^\perp$,
so $M_D=\pi(\exp \g_D) \subset M$ is a complete
flat totally geodesic submanifold. For any $g \in G$, let $\gamma=\gamma(s)$ be (the unique) geodesic realizing the
distance $d(g)$ from $\pi(g)$ to $M_D$, with $s \in [0,d(g)]$ an arclength parameter, such that
$\gamma(0)=\pi(g_D) \in M_D,\; \gamma(d(g))= \pi(g)$. Then $g_D^{-1}\gamma(s)g_D^{-1}$ is a geodesic of $M$ passing
through the basepoint $o=\pi(\id) \in M$ whose tangent vector is orthogonal to the tangent space of $M_D$ at $o$. It
follows that $g_D^{-1}\gamma(s) g_D^{-1} = \pi(\exp(sA))$ for some symmetric matrix $A \in \g_\phi$ such that
$Q(A,\g_D)=0$, that is, for some $A \in \p_0$. So $\pi(g)=g_D\exp(d(g)A) g_D=\pi(\exp(\frac12 d(g)A) g_D)$, as required.
\end{proof}
}

Assume that $\n$ is an Einstein nilradical and that $g_0 \in G_\phi$ is a critical point of the function $f$.
By Lemma~\ref{l:polar}, $g_0=upg_D$ for some $u \in U_\phi, \, p \in P_0, \, g_D \in G_D$.
As $f(ug)=f(g)$ for any $u \in U_\phi$ and $g \in G_\phi$, the point $g_1=pg_D$ is also critical.

Denote $\mu'=\exp(g_D).\mu$ and $\n' =(\Rn, \mu')$. The basis $\mathcal{B}$ is still a nice basis for $\n'$.
Moreover, by 1 of Lemma~\ref{l:tech}, $\phi$ is a pre-Einstein derivation of $\n'$, the algebra $\g_\phi$ and the
group $G_\phi$ for $\n'$ are the same as those for $\n$, and $f_{\n'}(g)=f_{\n}(gg_D)$ for $g \in G_\phi$.
We can therefore replace $\n$ by $\n'$ and
assume that the function $f$ has a critical point $p \in P_0$. Then $p=\exp A$ for some $A \in \p_0$. Consider a function
$F(t)=f(\exp(tA))$ for $t \in \br$. The function $F$ has a critical point at $t=1$, as $\exp(A)$ is critical for $f$,
and at $t=0$, as by 2(a) of Lemma~\ref{l:tech},
$F'(0) = \frac{d}{dt}_{|t=0} f(\exp(tA))= 2 \Tr (\ric_{\n} \;  A)=0$, since $\ric_{\n}$
is diagonal by \eqref{eq:ricdiag}. Let $\{E_i\}$ be an orthonormal basis of the eigenvectors of the symmetric matrix $A$,
with the corresponding eigenvalues $\nu_i$, and let $C_{ij}^k=\<[E_i,E_j],E_k\ra$ be the structural constants with respect
to the basis $\{E_i\}$. Then $F(t)=f(\exp(tA))=\sum_{i,j,k} (C_{ij}^k)^2 \exp(2t(\nu_i+\nu_j-\nu_k))$, so
$F''(t) \ge 0$. Since the convex function $F(t)$ has two critical points, it must be a constant, which implies
that $\nu_i+\nu_j=\nu_k$ for all the triples $(i,j,k)$ with $C_{ij}^k \ne 0$, so $A \in \Der(\n)$ and
$p \in \Aut(\n)$. As $f(g)= f(gp^{-1})$ for any $p \in \Aut(\n)$ and $g \in G_\phi$, the identity
is also a critical point for $f$.

Thus if $\n$ is an Einstein nilradical, then the function $f$ has a critical point on $\exp \g_D$, which proves
the ``only if" part of the theorem by assertion 1 of Lemma~\ref{l:nicediag}.
\end{proof}

The above proof shows that if a nilpotent Lie algebra with a nice basis is an Einstein nilradical, then the nilsoliton
inner product can be chosen diagonal. This is by no means obvious, as a nilpotent Lie algebra can have two quite
different nice bases, the easiest example being a direct sum of two copies of the Heisenberg algebra $\h_3$ given by
$[X_1,X_2]=X_5,\; [X_3,X_4]=X_6$. The basis $\{X_1 \pm X_3, X_2 \pm X_4, X_5 \pm X_6\}$ for this algebra is also nice,
with a different number of nonzero structural constants.

Also note that a nilsoliton inner product for a nilpotent Lie algebra with a nice basis can sometimes be found
explicitly.
For instance, as it follows from \eqref{eq:ricdiag} and \eqref{eq:preEdiag}, if $\rk \, Y = m$, the inner product
$\ip_r$ is nilsoliton, if we choose $r =(e^{s_1}, \ldots, e^{s_n}) \in \Rn$ in such a way that the $s_i$'s
satisfy the linear system $(Ys)_a= \sqrt{\alpha_a} \, |c_{ij}^k|^{-1}$ for $a=(i,j,k) \in \Lambda$.

Another application of Theorems~\ref{t:preE} and \ref{t:var} is the following theorem.

\begin{theorem} \label{t:rvsc}
Let $\n_1$ and $\n_2$ be two (real) nilpotent Lie algebras, whose complexifications are isomorphic as the complex
nilpotent Lie algebras. If $\n_1$ is an Einstein nilradical, then so is $\n_2$, with the same eigenvalue type.
\end{theorem}

Note that two real nilpotent algebras with isomorphic complexifications might be quite different. For instance,
two-step nilpotent algebras $\n_1$ and $\n_2$ defined by $[X_1,X_2]=Z_1,\, [X_3,X_4]=Z_2$,
and by $[X_1,X_3]=[X_2,X_4]=Z_1,\, [X_1,X_4]=[X_3,X_2]=Z_2$ respectively are isomorphic over $\bc$. However, the
algebra $\n_1$ is decomposable: it is a direct sum of two copies of the Heisenberg algebra $\h_3$, while $\n_2$ is
nonsingular (for any $X \in \n_2 \setminus \z$, where $\z$ is the center of $\n_2$, the map $\ad_X: \n_2 \to \z$ is
surjective \cite{Eb2}).

Theorem \ref{t:rvsc} can be useful when one knows the classification of
a family of nilpotent algebras only up to complex isomorphism
(see e.g. \cite{GT} or several lists of seven-dimensional nilpotent Lie algebras available in the literature).
Also, in the majority of the results of the Geometric Invariant Theory (which seems to be strongly present in the study
of Einstein nilradicals), the ground field is $\bc$ \cite{VP}.

\begin{proof}[Proof of Theorem~\ref{t:rvsc}]
We start with constructing a pre-Einstein derivation in the complex case, which we define as a semisimple derivation
$\phi$ satisfying \eqref{eq:pEtrace}. The proof follows the same lines as that of
assertion 1 of Theorem~\ref{t:preE} and shows that \emph{a pre-Einstein derivation always exists, is unique up to
conjugation by $\Aut(\n)$, and has all its eigenvalues rational}.

Let $\mathfrak{N}$ be a complex nilpotent Lie algebra, and let
$\Der (\mathfrak{N}) = \s \oplus \tg \oplus \n$ be the Levi-Mal'cev decomposition of $\Der (\mathfrak{N})$, where
$\tg \oplus \n$ is the radical of $\Der (\mathfrak{N})$, $\s$ is semisimple, $\n$ is the nilradical of $\tg \oplus \n$,
$\tg$ is an algebraic torus, and $[\tg, \s] = 0$.

The quadratic form $b$ on $\Der (\mathfrak{N})$
defined by $b(\psi_1, \psi_2) = \Tr (\psi_1  \psi_2)$ is
invariant and $b(\tg,\s)=b(\tg,\n)=0$. As $\Tr \psi = 0$ for any $\psi \in \s \oplus \n$, the derivation $\phi$ we are
after lies in $\tg$. The torus $\tg$ is fully reducible: in some basis for $\mathfrak{N}$, all the
elements from $\tg$ are given by diagonal matrices. As $\tg$ is algebraic, it is defined over $\mathbb{Q}$: there exists
a set of diagonal matrices with rational entries such that $\tg$ is their complex linear span. Explicitly, this follows
from the property, which is somewhat stronger then algebraicity.
Denote $D_{i}$ a matrix having $1$ as its $(i,i)$-th entry and zero elsewhere. For a nonzero $\psi \in \tg$, let
$\psi^\perp$ be the set of matrices $D=D_{i}+D_{j}-D_{k}$ such that $\Tr (\psi D) = 0$. Then $\tg$ contains all the
diagonal matrices $\psi'$ such that $\Tr (\psi' D) = 0$, for all $D \in \psi^\perp$ (the replicas of $\psi$).

The restriction of $b$ to $\tg$ is nondegenerate, and the matrix of $b$ is rational (and positive definite) with
respect to a rational basis for $\tg$. It follows that there exists a unique $\phi \in \tg$ such that \eqref{eq:pEtrace}
holds for all $\psi \in \tg$ (and hence for all $\psi \in \Der(\mathfrak{N})$), and the eigenvalues of $\phi$ are rational
numbers.
As in the proof of Theorem~\ref{t:preE}, the Mostow theorem \cite[Theorem 4.1]{Mos} implies that the pre-Einstein
derivation is unique up to conjugation by an automorphism of $\mathfrak{N}$.

Let now $\n=(\Rn, \mu)$ be a real nilpotent Lie algebra, with the complexification $\n^{\bc}=(\bc^n, \mu)$. Then
$\Der(\n^{\bc})=(\Der(\n))^{\bc}$, so the pre-Einstein derivation $\phi=\phi_{\n}$ also serves as a pre-Einstein
derivation for $\n^{\bc}$. It follows that the pre-Einstein derivations of two real Lie algebras whose complexifications
are isomorphic have the same eigenvalues.

By Theorem~\ref{t:var}, $\n$ is an Einstein nilradical if and only if the orbit $G_\phi.\mu$ is closed in $\mathcal{V}$.
Let $\g_\phi^\bc$ be the complexification of the Lie algebra $\g_\phi$ defined by \eqref{eq:subalgebra}, and let
$G^{\bc}_{\phi} \subset \SL(n, \bc)$ be the Lie group with the Lie algebra $\g_\phi^\bc$
(the group $G^{\bc}_{\phi}$ is defined by the right-hand side of \eqref{eq:tildegphi}, but over $\bc$).
Consider the orbit of $\mu$ in the space $\mathcal{V}^{\bc}=\wedge^2 (\bc^n)^* \otimes \bc^n$, the complexification of
$\mathcal{V}$, under the action of $G^{\bc}_{\phi}$. By the results of
\cite[Proposition 2.3]{BHC} and \cite[Corollary 5.3]{Bir}, the orbit $G_{\phi}.\mu$ is (Euclidean) closed in
$\mathcal{V}$ if and only if $G^{\bc}_{\phi}.\mu$ is Zariski-closed in $\mathcal{V}^{\bc}$.

It follows that two real algebras having isomorphic complexifications are or are not Einstein nilradicals
simultaneously. In the former case, the eigenvalue types are the same, as the spectra of the pre-Einstein derivations
are the same.
\end{proof}

\begin{remark} \label{rem:stable}
The proof shows that the property of a real nilpotent Lie algebra $\n$ to be an Einstein nilradical is, in fact, a
property of its complexification $\n^{\bc}$.
Namely, call a complex nilpotent Lie algebra $\mathfrak{N}=(\bc^n, \nu)$ with a pre-Einstein derivation $\phi$
\emph{stable}, if the orbit $G^{\bc}_{\phi}.\nu$ is Zariski-closed in $\mathcal{V}^{\bc}$.
Then $\n$ is an Einstein nilradical if and only if $\n^{\bc}$ is stable.
\end{remark}

Yet another application of Theorem~\ref{t:preE} and Theorem~\ref{t:var} is the following theorem.

\begin{theorem}\label{t:prod}
Let a nilpotent Lie algebra $\n$ be the direct sum of nilpotent Lie algebras $\n_1$ and $\n_2$. Then:

\begin{enumerate}[\rm (a)]
  \item
  If $\phi_i$ are pre-Einstein derivations for $\n_i$, then $\phi=\phi_1 \oplus \phi_2$ is
  a pre-Einstein derivation for $\n$.

  \item
  The algebra $\n$ is an Einstein nilradical if and only if both algebras $\n_1$ and $\n_2$ are.
\end{enumerate}

\end{theorem}

As it follows from \cite[Theorem 4]{Pay}, if both algebras $\n_1$ and $\n_2$ are Einstein nilradicals,
a nilsoliton inner product on $\n$ can be taken as the orthogonal sum of (appropriately scaled) nilsoliton inner
products on the $\n_i$'s. Geometrically this says that the Riemannian product of two Einstein solvmanifolds of
the same Ricci curvature is again an Einstein solvmanifold.

In the case when one of the summands is abelian, assertion (b) follows from \cite[Proposition~3.3]{La2}.
Note also that assertion (b) is very far from being true for the semidirect sum: even when both summands are
abelian, the resulting algebra $\n$ is two-step nilpotent and can easily be not an Einstein nilradical (see
Section~\ref{s:2step}).

\begin{proof}
In the proof, we use $\oplus$ in several different meanings: as a direct sum of linear spaces (even when the spaces
are Lie algebras), as the direct sum of operators, and as the direct sum of Lie brackets (in the obvious sense). We
use $\dotplus$ for the direct sum of Lie algebras.

(a) It is not difficult to see that the algebra $\Der (\n)$ admits the following splitting into four subspaces:
$\Der (\n) = \oplus_{i,j=1}^2 \dg_{ij}$, where $\dg_{ii}, \; i=1,2$, is the space of all $\psi \in \End(\n)$ such that
$\psi(\n_i) \subset \n_i$, $\psi_{|\n_i} \in \Der (\n_i), \; \psi_{|\n_j} = 0$, for $j \ne i$; and
$\dg_{ij}, \; i \ne j$, is the space of all $\psi \in \End(\n)$ such that
$\psi_{|\n_j \oplus [\n_i, \n_i]} = 0$ and $\psi(\n_i)$ lies in the center of $\n_j$.

Let now $\phi \in \dg_{11} \oplus \dg_{22}$ be the derivation of $\n$ such that $\phi_{|\n_i} = \phi_i, \; i=1,2$.
Clearly, $\phi$ is semisimple and real. Moreover, equation \eqref{eq:pEtrace} holds for any
$\psi \in \dg_{11} \oplus \dg_{22}$, as each of the $\phi_i$'s is pre-Einstein, and for any
$\psi \in \dg_{12} \oplus \dg_{21}$, as both sides vanish.

(b) The ``if" part follows directly from \cite[Theorem 4]{Pay}.
To prove the ``only if" part, choose and fix the pre-Einstein derivations
$\phi_1, \phi_2$, and $\phi=\phi_1 \oplus \phi_2$ for $\n_1, \n_2$, and $\n$ respectively and consider the algebras
$\g_{\phi_i} \subset \slg(\n_i), \; i=1,2$, and $\g_{\phi} \subset \slg(\n)$, as in \eqref{eq:subalgebra}. Note that
$\g_{\phi} \supset \g_{\phi_1} \oplus \g_{\phi_2}$. Let $\n_i=(\br^{n_i}, \mu_i), \; i=1,2$, then
$\n=(\br^{n_1+n_2}, \mu)$, with $\mu=\mu_1 \oplus \mu_2$.

Suppose that $\n$ is an Einstein nilradical, but $\n_1$ is not.
By assertion 2(a) of Theorem~\ref{t:varm}, there exists
$A_1 \in \g_{\phi_1} \subset \slg(\n_1)$ such that the limit $\mu_1'=\lim_{t\to\infty}(\exp(tA_1).\mu_1)$ exists, and
the algebra $\n_1'=(\br^{n_1}, \mu_1')$ is not isomorphic to $\n_1$. Then for $A=A_1 \oplus 0_{|\n_2} \in \g_{\phi}$,
we have
$\lim_{t\to\infty}(\exp(tA).\mu)=\mu_1'\oplus\mu_2$. As $\n$ is an Einstein nilradical, the algebra
$\n'=\n_1' \dotplus \n_2$ must be isomorphic to the algebra $\n=\n_1 \dotplus \n_2$ by (ii) of Theorem~\ref{t:var}.

The claim now follows from the purely Lie-algebraic fact that if the algebras $\n_1$ and $\n_1'$ are not isomorphic,
then so are $\n_1 \dotplus \n_2$ and $\n_1' \dotplus \n_2$, for any $\n_2$, which in turn follows from the uniqueness
of a decomposition of a Lie algebra into the direct sum of undecomposable ones, up to permutation and isomorphism.
To decompose an arbitrary (in particular, for a nilpotent) Lie algebra $\n$, one can use the following approach.
First, split off a direct abelian summand $\ag$, if it exists: $\n= \widehat{\n} \dotplus \ag$, where $\ag$ is a
linear complement to $[\n,\n]$ in the center $\z(\n)$, and $\widehat{\n} \supset [\n,\n]$ is a linear complement to
$\ag$ in $\n$.
This decomposition is not in general unique, but is unique up to a central automorphism of $\n$, namely up
to $h \in \Aut(\n)$ such that $(h-\id)(\n) \subset \z(\n)$ \cite[Theorem~2.4, Eq.~(2.28)]{RWZ}.
Next, when a particular $\widehat{\n}$ is chosen (note that $\z(\widehat{\n}) \subset [\widehat{\n},\widehat{\n}]$), the
further decomposition is unique up to permutation by \cite[Theorem~2.8(e)]{RWZ}. It follows that a decomposition of a Lie
algebra into the direct sum of undecomposable Lie algebras is unique, up to isomorphism and permutation, so the algebras
$\n_1 \dotplus \n_2$ and $\n_1' \dotplus \n_2$, with $\n_1 \not\simeq \n_1'$, are nonisomorphic.
\end{proof}

\section{Two-step Einstein nilradicals}
\label{s:2step}

In this section, the technique developed in the preceding sections is applied to the two-step nilpotent Lie algebras.
We prove Theorem~\ref{t:twostepopen} and also consider some exceptional cases.

We start with some preliminary facts, mostly following \cite{Eb2}.
A two-step nilpotent Lie algebra $\n$ of dimension $p + q$, is said to be \emph{of type} $(p,q)$, if its derived
algebra $\m = [\n, \n]$ has dimension $p$. Clearly, $\m \subset \z(\n)$, the center of $\n$, and
$1 \le p \le D:=\frac12 q (q-1)$.

Choose a subspace $\mathfrak{b}$ complementary to $\m$ in $\n$ and two bases: $\{X_i\}$ for $\mathfrak{b}$
and $\{Z_k\}$ for $\m$.
The Lie bracket on $\n$ defines (and is defined by) a $p$-tuple of skew-symmetric $q \times q$ matrices
$J_1, \ldots, J_p$ such that $[X_i, X_j] = \sum_{\al=1}^p (J_\al)_{ij} Z_\al$. The space of such $p$-tuples is
$\mathcal{V}(p,q)=(\wedge^2 \br^q)^p$. Note that the $J_\al$'s must be linearly independent, as $\m = [\n, \n]$,
so the points of $\mathcal{V}(p,q)$ corresponding to the algebras of type $(p,q)$ form a subset
$\mathcal{V}^0(p,q) \subset \mathcal{V}(p,q)$, which is the complement to a real algebraic subset. The spaces
$\mathcal{V}(p,q)$ and $\mathcal{V}^0(p,q)$ are acted upon by the group $\GL(q) \times \GL(p)$ (change of bases):
for $x = (J_1, \ldots, J_p) \in \mathcal{V}(p,q)$ and $(M, T) \in \GL(q) \times \GL(p)$,
$(M, T).x= (\tilde J_1, \ldots, \tilde J_p)$, with $\tilde J_\al=\sum_{\beta=1}^p (T^{-1})_{\beta\al} M J_\beta M^t$
(note that instead of fixing the basis and deforming the Lie bracket, as in the action
$g.\mu(X,Y)=g\mu(g^{-1}X,g^{-1}Y)$, we now keep the Lie bracket fixed and change the basis).
Clearly, two points of $\mathcal{V}^0(p,q)$ lying on the same $\GL(q) \times \GL(p)$-orbit define isomorphic algebras.
The converse is also true, so that the space $\mathcal{X}(p, q)$ of the isomorphism classes of two-step nilpotent
Lie algebras of type $(p, q)$ is the quotient space $\mathcal{V}^0(p,q)/(\GL(q) \times \GL(p))$.
The space $\mathcal{X}(p, q)$ is compact, but in general is non-Hausdorff.

By Proposition A of \cite[Section 5.4d]{Eb1},
when $q \ge 6$ and $2 < p < D-2$, or when $(p,q) = (5,5)$, the space $\mathcal{V}(p,q)$ contains no
open $\GL(q) \times \GL(p)$-orbits, that is, no two-step nilpotent algebras of type
$(p,q)$ are locally rigid (an open orbit of the action of $\SL(q)$ on the Grassmannian
$G(p, \wedge^2 \br^q)$ considered in \cite{Eb1} occurs exactly when the action of $\GL(q) \times \GL(p)$ on
$\mathcal{V}(p,q)$ has an open orbit).

It will be more convenient to consider the action of $\SL(q) \times \SL(p)$, rather than $\GL(q) \times \GL(p)$.
The $\SL(q) \times \SL(p)$-orbits distinguish the isomorphism classes up to scaling, which is
easy to control.

The splitting $\n = \mathfrak{b} \oplus \m$ of a two-step nilpotent Lie algebra $\n$ is a gradation,
which corresponds to the
\emph{canonical derivation} $\Psi$ defined by $\Psi (X + Z) = X + 2 Z$, for any $X \in \mathfrak{b},\; Z \in \m$.
If a pre-Einstein derivation $\phi$ for $\n$ is proportional to $\Psi$,
then the group $G_\phi$ from Theorem~\ref{t:var} is precisely $\SL(q) \times \SL(p)$ (see \eqref{eq:Ggphi}).
If, in addition, $\n$ is an Einstein nilradical, then it has the eigenvalue type $(1,2; q, p)$.
Thus Theorem~\ref{t:var} implies the following:

\begin{proposition}[{\cite[Proposition 9.1]{La3}}] \label{p:type12}
A two-step nilpotent Lie algebra $\n$ of type $(p,q)$ corresponding to a point $x \in \mathcal{V}(p,q)$ is an Einstein
nilradical of the eigenvalue type $(1,2; q, p)$ if and only if the orbit
$(\SL(q) \times \SL(p)).x \subset \mathcal{V}(p,q)$ is closed.
\end{proposition}

Note that if a pre-Einstein derivation of a nilpotent Lie algebra $\n$ is proportional to the one having only eigenvalues
$1$ and $2$, then $\n$ is automatically two-step nilpotent. Theorem~\ref{t:twostepopen} says that a typical algebra with
such pre-Einstein derivation is an Einstein nilradical.

Let $\n$ be a two-step nilpotent Lie algebra of type $(p,q)$, with $1 \le p < D$, defined by a point
$x = (J_1, \ldots, J_p) \in \mathcal{V}^0(p,q)$. Choose an arbitrary basis $J'_\al, \; \al= 1, \ldots, D-p$, in the
orthogonal complement to the subspace $\Span(J_1, \ldots, J_p) \subset \wedge^2 \br^q$ with respect to the inner product
$Q(K_1, K_2)= -\Tr (K_1 K_2)$ on $\wedge^2 \br^q$. The point $x'=(J_1', \ldots, J_{D-p}') \in \mathcal{V}^0(D-p,q)$
defines a two-step nilpotent Lie algebra $\n^*$ of type $(D-p,q)$, which is called the \emph{dual} to $\n$. It is
easy to check that the isomorphism class of $\n^*$ is well-defined (depends only on the isomorphism class of $\n$).

\begin{proof}[Proof of Theorem~\ref{t:twostepopen}]
Let a pair $(p,q)$ with $1 \le p \le q(q-1)/2$,
be such that the stabilizer in general position (the s.g.p.) of the group
$G^{\bc}=\SL(q,\bc) \times \SL(p,\bc)$ acting on the space
$\mathcal{V}^{\bc}(p,q)=(\wedge^2 \bc^q)^p$ is reductive. This means that there exists a reductive complex Lie algebra
$\h$ and a nonempty $G^{\bc}$-invariant Zariski-open subset $U' \subset \mathcal{V}^{\bc}(p,q)$ such that for any point
$x \in U'$ the Lie algebra of the stabilizer $G^{\bc}_x$ is isomorphic to $\h$ \cite[\S 7]{VP}.

By \cite{Ela}, in the most cases, the s.g.p. of the action of $G^{\bc}$ on $\mathcal{V}^{\bc}(p,q)$ is finite, that is,
$\h=0$ (as it is suggested by the dimension count). The cases when it is not are listed in Table 6 of \cite{Ela}.
Examining that table we see that the s.g.p. is reductive unless $(p,q) = (2, 2k+1)$.

By the Popov criterion \cite{Pop}, in all the cases when the s.g.p. is reductive, the action is stable, that is, there
exists a nonempty Zariski-open subset $U'' \subset \mathcal{V}^{\bc}(p,q)$, which is a union of
closed $G^{\bc}$-orbits. Let $U$ be the set of real points of $U'' \cap U'$. Then $U$ is
a $\SL(q) \times \SL(p)$-invariant semialgebraic subset of
$\mathcal{V}(p,q)=(\wedge^2 \br^q)^p$ and is open and dense in the Euclidean topology in $\mathcal{V}(p,q)$.

It now follows from Remark~\ref{rem:stable} and Proposition~\ref{p:type12} that for every point $x \in U, \; x \ne 0$,
the two-step nilpotent Lie algebra corresponding to $x$ is an Einstein nilradical.
Note that every such $x$ lies in $\mathcal{V}^0(p,q)$, hence defining a two-step nilpotent Lie algebra $\n$
precisely of type $(p,q)$ (for if $x =(J_1, \ldots, J_p) \in U$ and the $J_\al$'s are linearly dependent, then the
closure of the orbit $(\SL(q) \times \SL(p)).x$ contains the origin of $\mathcal{V}(p,q)$).
Hence the eigenvalue type of the Einstein nilradical $\n$ is $(1,2;q,p)$.
\end{proof}

The proof shows that a generic point of $\mathcal{V}(p,q)$ defines an Einstein nilradical with the eigenvalue
type $(1,2;q,p)$ in all the cases except for $(p,q)=(2, 2k+1)$ (in fact, there are no two-step
Einstein nilradicals of type $(1,2;2k+1,2)$ at all, see \cite[Remark~3]{Ni1}). In Theorem~\ref{t:twostepopen}, we narrow
the dimension range
to exclude those cases when some algebras of type $(p,q)$ have open orbits in $\mathcal{V}(p,q)$ (and additionally,
the cases  $(p,q)=(2, 2k),\; k >3$). The remaining cases give a reasonable notion of being
typical not only in the linear space $\mathcal{V}(p,q)$, but also in the non-Hausdorff space $\mathcal{X}(p,q)$ of
isomorphism classes of two-step nilpotent algebras of type $(p,q)$. Note that in general, the condition of typicality
of a nilpotent Lie algebra could be hardly nicely defined (see \cite{Luk}). On the other hand, taking the categorical
quotient $\mathcal{V}(p,q)/\!/(\SL(q) \times \SL(p))$ is somewhat tautological in view of the proof of
Theorem~\ref{t:twostepopen}.

Theorem~\ref{t:twostepopen} omits two-step nilpotent Lie algebras $\n$ with the following $(p,q)$:
\begin{itemize}
    \item $p=1$. Any such $\n$ is the direct sum of a Heisenberg algebra and an abelian ideal and is an
    Einstein nilradical (the corresponding solvmanifold can be taken as the product of a real and a complex hyperbolic
    space).
    \item $p=D$, a free two-step nilpotent algebra; $\n$ is an Einstein nilradical by
    \cite[Proposition 2.9]{GK}.
    \item $p=D-1$; any such algebra is an Einstein nilradical by \cite[Lemma 6]{Ni1}.
    \item $p=2, \; p=D-2$. These algebras can be completely classified using the Kronecker theory of matrix pencils;
    the approach suggested at the end of Section~\ref{s:thpE} can be used to find Einstein nilradicals among them.
    \item $q \le 5, \; (p,q) \ne (5,5)$. Using Theorem~\ref{t:nice} and the classification from \cite{GT} we find below
    all the Einstein nilradicals among these algebras.
\end{itemize}

First of all, by \cite[Theorem 3.1]{Wil} and \cite[Theorem 5.1]{La2}, any nilpotent
Lie algebra of dimension six or lower is an Einstein nilradical (no longer true in dimension $7$). Secondly, any
two-step nilpotent algebra with $p=1, D$, or $D-1$, is an Einstein nilradical. This leaves out the following list of
pairs $(p,q)$:
\begin{equation*}
(p,q)= (3,4), \; (4,4), \; (2,5), \; (3,5), \; (4,5), \; (6,5), \; (7,5), \; (8,5).
\end{equation*}
According to the classification
in \cite[Table 2]{GT}, all the algebras with $(p,q)= (3,4), \; (4,4), \; (2,5), \; (3,5)$, and $(4,5)$ over $\bc$ fall
into a finite number of classes, each of which is defined over $\br$. The same is true for the remaining three
cases $(p,q)= (6,5), \; (7,5), \; (8,5)$, which give the algebras dual to the algebras of types
$(p,q)= (4,5), \; (3,5)$, and $(2,5)$ respectively. Note that the dual to an Einstein nilradical is not necessarily
an Einstein nilradical.

By Theorem~\ref{t:rvsc}, it suffices to consider the algebras in \cite[Table 2]{GT} and their duals as the
\emph{real} Lie algebras.
The majority of them (all five of type $(3,4)$, all three of type $(4,4)$, all five of type $(2,5)$,
$15$ out of $17$ of type $(3,5)$, and $31$ out of $38$ of type $(4,5)$) have a nice basis, which is the basis given in
the table. Moreover, as $q \le 5$, the two-step nilpotent Lie algebra dual to a one having a nice basis also has a nice
basis, which easily follows from Definition~\ref{d:nice}.

\begin{table}[htbp]
\renewcommand{\arraystretch}{1.2}
\begin{tabular}{|p{.4cm}|p{4.7cm}|p{9.42cm}|}
  \hline
  $\#$ & \centering Relations & \centering Orthonormal basis \tabularnewline
  \hline
  \textbf{26} & 133, 152, 233, 244, 251, 341 & $\frac{-1}{\sqrt3}(2X_1-X_2),\frac{\sqrt2}{\sqrt3}X_2, X_3,X_4,X_5,
  Z_3,Z_4,\frac{-\sqrt{3}}{\sqrt{10}}(Z_1+2Z_2),\frac{\sqrt3}{\sqrt2}Z_2$ \\
  \textbf{28} & \hskip-.1cm 134, 143, 152, 233, 242, 251, 342 & $2\sqrt{3}X_1,-2X_2+X_3, \sqrt{6}X_3,\sqrt{6}X_4,X_5,
  Z_4,Z_3,\frac{1}{\sqrt{2}}Z_2,\frac{1}{\sqrt{30}}Z_1$\\
  \textbf{44} & 124, 143, 152, 232, 242, 351 & $\sqrt{6}X_1,\sqrt{6}X_2, -2X_3+X_4,\sqrt{6}X_4,X_5,
  Z_4,Z_3,Z_2,\frac{1}{\sqrt{15}}Z_1$\\
  \textbf{45} & 123, 142, 151, 232, 243, 344 & $\frac{1}{\sqrt6}(2X_1+X_4),2X_2, X_3,X_4,\sqrt{\frac52}X_5,
  Z_4,2Z_3,2Z_2,Z_1$ \\
  \textbf{55} & 124, 132, 142, 243, 351 & $X_1,X_2, \frac{1}{\sqrt6}(-2X_3+X_4),X_4,X_5,
  Z_4,Z_3,Z_2,\sqrt{\frac25}Z_1$ \\
  \textbf{60} & 124, 132, 143, 232, 251, 341 & $\frac{1}{\sqrt2}(-2X_1+X_2),\sqrt{\frac32}X_2,X_3,X_4,X_5,
  \sqrt{3}Z_4,Z_3,Z_2,Z_1$ \\
  \textbf{66} & 124, 131, 153, 231, 242 & $-2X_1+X_2,\sqrt{3}X_2,X_3,X_4,X_5, \sqrt{12}Z_4,Z_3,Z_2,Z_1$ \\
  \textbf{72} & 132, 143, 232, 251, 341 & $X_1+X_2, \sqrt{3} (X_1-X_2), X_3,X_4,X_5,Z_3,Z_1,Z_2$ \\
  \textbf{78} & 131, 153, 231, 242 & $X_1+X_2, \sqrt{3} (X_1-X_2), X_3,X_4,X_5, Z_1,Z_2,Z_3$ \\
  \hline
\end{tabular}
\vskip .25cm
\caption{Two-step nilpotent Lie algebras of types $(3,5)$ and $(4,5)$ having no nice basis}\label{table:nobasis}
\end{table}

Therefore, for all but two algebras of type $(3,5)$, seven algebras of type $(4,5)$ and their duals (two of type
$(7,5)$ and seven of type $(6,5)$), the question of whether the algebra is an Einstein nilradical is completely answered
by Theorem~\ref{t:nice}. It turns out that in all these cases the matrix $Y$ has the maximal rank $m$, so that the
equation $YY^t \al = [1]_m$ from (ii) of Theorem~\ref{t:nice} has a unique solution $\al$. A direct
computation shows that the only cases when the algebra fails to be an Einstein nilradical are the following:
three algebras of type $(7, 5)$ dual to the algebras $\mathbf{75},\mathbf{87}$ and $\mathbf{102}$, and six algebras
of type $(6, 5)$ dual to the algebras $\mathbf{21},\mathbf{36},\mathbf{41},\mathbf{50},\mathbf{52}$ and $\mathbf{59}$
from \cite[Table 2]{GT}.
Perhaps, the most interesting of them is algebra $\mathbf{102}$: viewed as an algebra of type $(3,5)$, it is the direct
sum of the free two-step nilpotent algebra on three generators and a two-dimensional abelian ideal.

\begin{table}[htbp]
\renewcommand{\arraystretch}{1.2}
\begin{tabular}{|p{.6cm}|p{4.8cm}|p{9.11cm}|}
  \hline
  $\#$ & \centering Relations & \centering Orthonormal basis \tabularnewline
  \hline
  $\mathbf{26}^*$ & $-251,341,-132,232,123,144$, $355,456$ &
  $-\frac{\sqrt3}{\sqrt{91}} (26 X_1+ 19 X_2),-\frac{3\sqrt{209}}{\sqrt{91}} X_2,\frac{12\sqrt{19}}{7} X_3,
  \frac{11}{6} X_4, \frac{\sqrt{22}}{\sqrt3} X_5$, \newline
  $Z_1, \frac{6\sqrt6}{\sqrt{77}} Z_2, Z_3, \frac{\sqrt{627}}{16\sqrt{14}} Z_4, Z_5, \frac{77}{72\sqrt{19}} Z_6$ \\
  $\mathbf{28}^*$ & $121,-153-154,232,-142,243$, $344,355$, $456$ &
  $\frac{8\sqrt{266}}{\sqrt{5}} X_1, \frac{12\sqrt{57}}{\sqrt{145}} X_2, \frac{1}{\sqrt{29}} (9 X_2 + 29 X_3), X_4,
  \frac{\sqrt{3}}{2\sqrt{14}} X_5, \frac{16\sqrt{14}}{5} Z_1$, \newline
  $\frac{4\sqrt{2}}{\sqrt{5}} Z_2,
  \frac{4\sqrt{6}}{\sqrt{145}} Z_3, \frac{1}{19\sqrt{29}} (-9 Z_3 + 29 Z_4),
  \frac{\sqrt{3}}{10\sqrt{7}} Z_5, \frac{\sqrt{3}}{16\sqrt{266}} Z_6$ \\
  $\mathbf{44}^*$ & $131,-152-153,232,243,254$, $345,456$ &
  $\frac{2\sqrt{199}}{3\sqrt{145}}X_1,\frac{1}{\sqrt{29}} X_2, \frac{\sqrt{2}}{3\sqrt{2805}}(X_3+\frac{33}{58}X_4),
  \frac{1}{29}X_4, \frac{1}{\sqrt{5771}} X_5,Z_6$, \newline
  $\frac{398\sqrt{2}}{3\sqrt{33}} Z_1, \frac{10\sqrt{34}}{\sqrt{33}} Z_2, \frac{-\sqrt{65}}{78}(154 Z_2+199 Z_3),
  \frac{5\sqrt{17}}{\sqrt{23}} Z_4, \frac{2\sqrt{995}}{\sqrt{1353}} Z_5$ \\
  $\mathbf{45}^*$ & $-141,231,-122,242,133,254$, $355,456$ &
  $13X_1-8X_4,\frac25 X_2,260 X_3, 12 X_4,X_5, \sqrt{130}Z_1, \frac{1}{\sqrt{10}} Z_2$, \newline
  $65 \sqrt{10} Z_3, \frac{1}{10 \sqrt{13}}Z_4, 5\sqrt{13}Z_5, Z_6$ \\
  $\mathbf{55}^*$ & $-131,141,152,233,254,345$, $456$ &
  $X_1,X_2, \frac{1}{2\sqrt6}(17X_3+12X_4),\sqrt{11}X_4, X_5, \frac{\sqrt{34}}{2\sqrt3} Z_1,
  \frac{\sqrt{5}}{\sqrt7} Z_2$, \newline
  $\frac{\sqrt{34}}{2\sqrt3} Z_3, \frac{\sqrt{5}}{\sqrt7} Z_4, \frac{17\sqrt{55}}{2\sqrt{39}} Z_5, \sqrt{17} Z_6$ \\
  $\mathbf{60}^*$ & $-251,341,-132,232,153,244$, $355,456$ &
  $\sqrt{19}(2 X_1+X_2),3\sqrt3 X_2,4X_3, \sqrt{33}X_4,\sqrt{22}X_5,
  Z_1, \frac{4}{\sqrt{33}} Z_2$, \newline
  $\frac{\sqrt{19}}{\sqrt{51}}(-Z_1+2Z_3), Z_4, \frac{4}{\sqrt{33}} Z_5, Z_6$ \\
  $\mathbf{66}^*$ & \hskip-.1cm $-131,231,452,143,254,345,356$ &
  $ \frac{1}{\sqrt{2}} (2X_1+X_2), \frac{\sqrt3}{\sqrt2} X_2,X_3,X_4,X_5,Z_1,Z_2,Z_3,Z_4,Z_5,Z_6$ \\
  $\mathbf{72}^*$ & $-131,231,-252,342,123,154$, $245,356,457$ &
  $4X_1+2X_2, \frac{\sqrt{19}}{\sqrt2}X_2, \frac13 X_3, \sqrt{11}X_4,\frac{1}{\sqrt{57}} X_5$,
  $Z_1, Z_2, \frac{12\sqrt{19}}{\sqrt{29}} Z_3$, \newline
  $\frac{2\sqrt{6}}{5\sqrt{19}} (Z_2-2Z_4),
  3\sqrt{11} Z_5, \frac{\sqrt{2}}{\sqrt{1311}} Z_6, \frac{\sqrt{66}}{\sqrt{437}} Z_7$ \\
  $\mathbf{78}^*$ & $-131,231,122,143,254,345,$ $356,457$ &
  $-2\sqrt{3} (X_1+X_2), 2(X_1-X_2), \sqrt{14}X_3,\sqrt{51}X_4,\sqrt{51}X_5$, \newline
  $\frac{\sqrt{14}}{\sqrt{51}} Z_1, \frac{4\sqrt{2}}{3\sqrt{17}} Z_2,
  Z_3, Z_4, Z_5, Z_6, \frac{\sqrt{51}}{\sqrt{14}} Z_7$
    \\
  \hline
\end{tabular}
\vskip .25cm
\caption{The algebras dual to the algebras from Table~\ref{table:nobasis}}\label{table:nobasisdual}
\end{table}

In the remaining cases, when there is no nice basis, we follow steps (\ref{it:a}) and (\ref{it:c})
at the end of Section~\ref{s:thpE},
first finding a pre-Einstein derivation and then solving the system of equations for the nilsoliton metric.
All nine algebras
$\mathbf{26}, \mathbf{28}, \mathbf{44}, \mathbf{45}, \mathbf{55}, \mathbf{60},$ $\mathbf{66}, \mathbf{72}, \mathbf{78}$
and their duals appear to be Einstein nilradicals. An explicit form of the nilsoliton inner product is given
in Tables~\ref{table:nobasis} and \ref{table:nobasisdual}, where
we use the following notation. The first column is the
number of the algebra in the list \cite[Table 2]{GT}, the asterisk means the dual algebra. The column ``Relations"
lists the nonzero brackets with respect to the basis $\{X_i, Z_\al\}$. For instance,
algebra $\mathbf{28}^*$ is defined by the relations
$[X_1,X_2]=Z_1$, $[X_1,X_5]=-Z_3-Z_4, \, [X_2,X_3]=Z_2$, $[X_1,X_4]=-Z_2, \, [X_2,X_4]=Z_3, \,
[X_3,X_4]=Z_4, \, [X_3,X_5]=Z_5$, and $[X_4,X_5]=Z_6$. The third column gives an orthonormal basis for a nilsoliton
inner product (which can be checked directly by \eqref{eq:riccinilexplicit} and \eqref{eq:ricn}).
Note that the algebras from Tables~\ref{table:nobasis} and~\ref{table:nobasisdual}
exhibit quite exotic eigenvalue types.

Summarizing these results we obtain the following proposition.

\begin{proposition}\label{p:lowdim}
A two-step nilpotent Lie algebra of type $(p,q)$, with $q \le 5$ and $(p,q) \ne (5,5)$, is an Einstein nilradical,
unless it is isomorphic (over $\bc$) to an $11$-dimensional algebra of type $(6,5)$ dual to one of the algebras
$\mathbf{21},\mathbf{36},\mathbf{41},\mathbf{50},\mathbf{52},\mathbf{59}$, or to a $12$-dimensional algebra of
type $(7,5)$ dual to one of the algebras $\mathbf{75},\mathbf{87}$, or $\mathbf{102}$ from \emph{\cite[Table 2]{GT}}.
\end{proposition}


\bibliographystyle{amsalpha}

\begin{thebibliography}{BHC}

\bibitem[Al1]{Al1}
Alekseevskii D.V.,
\emph{Classification of quaternionic spaces with transitive solvable group of motions},
Izv. Akad. Nauk SSSR Ser. Mat., \textbf{39} (1975), 315 -- 362. MR0402649 (53: 6465)

\bibitem[Al2]{Al2}
\bysame,
\emph{Homogeneous Riemannian spaces of negative curvature},
Mat. Sb. (N.S.) \textbf{96(138)} (1975), 93 -- 117. MR0362145 (50:14587)

\bibitem[AK]{AK}
Alekseevskii D.V., Kimel'fel'd B.N.,
\emph{Structure of homogeneous Riemannian spaces with zero Ricci curvature},
Funct. Anal. Appl., \textbf{9} (1975), 97 -- 102. MR0402650 (53:6466)

\bibitem[Bir]{Bir}
Birkes D.,
\emph{Orbits of linear algebraic groups},
Ann. of Math. (2), \textbf{93} (1971), 459 -- 475.
MR0296077 (45:5138)

\bibitem[BHC]{BHC}
Borel A., Harish-Chandra,
\emph{Arithmetic Subgroups of Algebraic Groups},
Ann. of Math. (2), \textbf{75} (1962), 485 -- 535.
MR0147566 (26:5081)

\bibitem[DM]{DM}
Dotti Miatello I.,
\emph{Ricci curvature of left-invariant metrics on solvable unimodular Lie groups},
Math. Z., \textbf{180} (1982), 257 -- 263. MR0661702 (84a:53044)

\bibitem[Eb1]{Eb1}
Eberlein P.,
\emph{The moduli space of 2-step nilpotent Lie algebras of type $(p,q)$},
Explorations in complex and Riemannian geometry, 37--72, Contemp. Math., 332, Amer. Math. Soc., Providence, RI, 2003,
37 -- 72. MR2016090 (2004j:17015)

\bibitem[Eb2]{Eb2}
\bysame,
\emph{Geometry of 2-step nilpotent Lie groups},
Modern Dynamical Systems and applications, Cambridge Univ. Press, Cambridge, 2004,
67 -- 101. MR2090766 (2005m:53081)

\bibitem[Eb3]{Eb3}
\bysame,
\emph{Riemannian $2$-Step nilmanifolds with prescribed Ricci tensor},
AMS volume in honor of the 60th birthday of
Misha Brin, to appear.

\bibitem[Ela]{Ela}
Elashvili A.G.,
\emph{Stationary subalgebras of points of general position for irreducible linear Lie groups},
Funct. Anal. Appl., \textbf{6} (1972),
139 -- 148. MR0304555 (46: 3690)

\bibitem[GK]{GK}
Gordon C., Kerr M.,
\emph{New Homogeneous Einstein Metrics of Negative Ricci Curvature},
Ann. Global Anal. Geom.,  \textbf{19} (2001), 75 -- 101.
MR1824172 (2002f:53067)

\bibitem[GT]{GT}
Galitski L., Timashev D.,
\emph{On classification of metabelian Lie algebras},
J. Lie Theory, \textbf{9} (1999), 125~--~156.
MR1680007 (2000f:17015)

\bibitem[Heb]{Heb}
Heber J.,
\emph{Noncompact homogeneous Einstein spaces},
Invent. Math., \textbf{133} (1998), 279 -- 352. MR1632782 (99d:53046)

\bibitem[La1]{La1}
Lauret J.,
\emph{Ricci soliton homogeneous nilmanifolds},
Math. Ann., \textbf{319} (2001), 715 -- 733. MR1825405 (2002k:53083)

\bibitem[La2]{La2}
\bysame,
\emph{Finding Einstein solvmanifolds by a variational method},
Math. Z., \textbf{241} (2002), 83~--~99. MR1930986 (2003g:53064)

\bibitem[La3]{La3}
\bysame,
\emph{Degenerations of Lie algebras and geometry of Lie groups},
Differential Geom. Appl., \textbf{18} (2003), 177 -- 194. MR1958155 (2004c:17008)

\bibitem[La4]{La4}
\bysame,
\emph{On the moment map for the variety of Lie algebras},
J. Funct. Anal., \textbf{202} (2003), 392~--~423. MR1990531 (2004d:14067)

\bibitem[La5]{La5}
\bysame,
\emph{Einstein solvmanifolds are standard}, preprint,
2007, arXiv: math/0703472.

\bibitem[LW]{LW}
Lauret J., Will C.,
\emph{Einstein solvmanifolds: existence and non-existence questions}, preprint
2006, arXiv: math.DG/0602502.

\bibitem[Luk]{Luk}
Luks E.M.,
\emph{What is the typical nilpotent Lie algebra?}
Computers in nonassociative rings and algebras, 189 -- 207.
Academic Press, New York, 1977. MR0453830 (56:12083)

\bibitem[Mil]{Mil}
Milnor J.,
\emph{Curvatures of left invariant metrics on Lie groups},
Adv. Math. \textbf{21} (1976), 293 -- 329. MR0425012 (54:12970)

\bibitem[Mos]{Mos}
Mostow G.D.,
\emph{Fully reducible subgroups of algebraic groups},
Amer. J. Math., \textbf{78} (1956), 200~--~221. MR0092928 (19,1181f)

\bibitem[Ni1]{Ni1}
Nikolayevsky Y.,
\emph{Nilradicals of Einstein solvmanifolds}, preprint,
2006, arXiv: math.DG/0612117.

\bibitem[Ni2]{Ni2}
\bysame,
\emph{Einstein solvmanifolds with a simple Einstein derivation}, preprint,
2007, arXiv: math.DG/0707.4595.

\bibitem[Ni3]{Ni3}
\bysame,
\emph{Einstein solmanifolds with free nilradical},
Ann. Global Anal. Geom., \textbf{33} (2008), 71 -- 87. MR2369187

\bibitem[Pay]{Pay}
Payne T.,
\emph{The existence of soliton metrics for nilpotent Lie groups},
preprint, 2005.

\bibitem[Pop]{Pop}
Popov V. L.,
\emph{Criteria for the stability of the action of a semisimple group on the factorial of a manifold},
Izv. Akad. Nauk SSSR Ser. Mat., \textbf{34} (1970), 523 -- 531. 
MR0262416 (41:7024)

\bibitem[RS]{RS}
Richardson R., Slodowy P.,
\emph{Minimum vectors for real reductive algebraic groups},
J. London Math. Soc. (2), \textbf{42} (1990), 409 -- 429. MR1087217 (92a:14055)

\bibitem[RWZ]{RWZ}
Rand D., Winternitz P., Zassenhaus H.,
\emph{On the identification of a Lie algebra given by its structure constants. I. Direct decompositions,
Levi decompositions, and nilradicals},
Linear Algebra Appl., \textbf{109} (1988), 197 -- 246. MR0961578 (89i:17001)

\bibitem[VP]{VP}
Vinberg E.B., Popov V.L.,
\emph{Invariant theory}. In: Algebraic geometry. IV. Encyclopaedia of Mathematical Sciences, 55,
Springer-Verlag, Berlin, 1994. MR1100485 (92d:14010)

\bibitem[Wil]{Wil}
Will C.,
\emph{Rank-one Einstein solvmanifolds of dimension 7},
Differential Geom. Appl., \textbf{19} (2003), 307~--~318. MR2013098 (2004j:53060)

\bibitem[WZ]{WZ}
Wolf J.A., Zierau R.,
\emph{Riemannian exponential maps and decompositions of reductive Lie groups},
Topics in geometry, Progr. Nonlinear Differential Equations Appl., 20, Birkhäuser Boston, Boston, 1996,
349--353. MR1390323 (97c:53080)

\end{thebibliography}

\end{document}